\documentclass[11pt, a4paper]{amsart}
\usepackage{amsthm, amssymb, mathtools, cite, graphicx, color,, subcaption, amsmath}
\mathtoolsset{showonlyrefs}
\newtheorem{theorem}{Theorem}[section]
\newtheorem*{theorem*}{Theorem}

\newtheorem{proposition}[theorem]{Proposition}
\newtheorem{corollary}[theorem]{Corollary}
\newtheorem*{corollary*}{Corollary}
\newtheorem{definition}[theorem]{Definition}

\numberwithin{equation}{section}
\title{Entropy, Minimal Surfaces and negatively curved manifolds}
\author{Andrew Sanders}

\begin{document}

\address{Department of Mathematics, Statistics and Computer Science, University of Illinois at Chicago, Chicago, IL 60607 USA}
\email{andysan@uic.edu}
\thanks{Sanders gratefully acknowledges partial support from the National Science Foundation Postdoctoral Research Fellowship and from U.S. National Science Foundation grants DMS 1107452, 1107263, 1107367 "RNMS: GEometric structures And Representation varieties" (the GEAR Network).}


\date{October 23,2013}


\begin{abstract}
In \cite{TAU04}, Taubes introduced the space of minimal hyperbolic germs with elements consisting of the first and second fundamental form of an equivariant immersed minimal disk in hyperbolic 3-space.  Herein, we initiate a further study of this 
space by studying the behavior of a dynamically defined function which records the entropy of the geodesic flow on the associated Riemannian surface.  We provide a useful estimate on this function which, in particular, yields a new proof of Bowen's theorem on the rigidity of the Hausdorff dimension of the limit set of quasi-Fuchsian groups.  These follow from new lower bounds on the Hausdorff dimension of the limit set which allow us to give a quantitative version of Bowen's rigidity theorem.  To demonstrate the strength of the techniques, these results are generalized to convex-cocompact surface groups acting on $n$-dimensional CAT$(-1)$ Riemannian manifolds.
\end{abstract}
\maketitle

\section{Introduction}

Given a convex-cocompact hyperbolic 3-manifold $M$ and a $\pi_1$-injective mapping $f:\Sigma\rightarrow M$ of a closed surface $\Sigma$ into $M,$ the general existence and regularity theory developed by Meeks-Simon-Yau \cite{MSY82}, Sacks-Uhlenbeck \cite{SU82}, Freedman-Hass-Scott \cite{FHS83} and Osserman-Gulliver \cite{GUL77} furnishes the existence of an immersed minimal surface $\Sigma\rightarrow M$ in the homotopy class of $f$ which minimizes area among all maps in the homotopy class.  Motivated by this proliferation of closed minimal surfaces in hyperbolic 3-manifolds, Taubes \cite{TAU04} constructed the space of minimal hyperbolic germs $\mathcal{H}$ which is a deformation space whose typical element consists of a Riemannian metric and symmetric 2-tensor $(g,B)$ which together are the induced metric and second fundamental form of a minimal immersion of $\Sigma$ into a potentially incomplete hyperbolic 3-manifold.  The present paper is the beginning of a deeper investigation of the space $\mathcal{H}$ and its relationship to the $\text{PSL}(2,\mathbb{C})$-character variety $\mathcal{R}(\pi_1(\Sigma),\text{PSL}(2,\mathbb{C})).$

We begin $\S$\ref{sec: entropy} with a study of a dynamically defined function on the space of minimal hyperbolic germs $\mathcal{H}.$  A pair $(g,B)\in\mathcal{H}$ satisfies a trio of equations, the most important of which for us is the \textit{Gauss} equation:
\begin{align}
K_g=-1-\frac{1}{2}\lVert B\rVert_g^2.
\end{align}
Here $K_g$ is the sectional curavture of the metric $g.$  In particular, all elements of the space $\mathcal{H}$ are Riemannian surfaces whose sectional curvature is bounded above by $-1.$  We define a function,
\begin{align}
E:\mathcal{H}\rightarrow \mathbb{R},
\end{align}
which records the topological entropy of the geodesic flow on the unit tangent bundle of the Riemannian surface $(\Sigma,g).$  Our analysis shows that $E$ is greater than or equal to $1,$ and the subspace along which it is equal to $1$ is precisely the Fuchsian space $\mathcal{F}$ of constant negative curvature metrics on $\Sigma.$  The (non-Riemannian) hessian of $E$ along $\mathcal{F}$ yields a well-defined non-negative symmetric bilinear form on the tangent space to $\mathcal{F}$ which is shown to be bounded below by the Weil-Petersson metric.  Hence, the Hessian of $E$ yields a metric on $\mathcal{F}$ which is invariant under the mapping class group of $\Sigma.$  We also identify an open subset of  minimal hyperbolic germs along which $E$ has no critical points; the almost-Fuchsian germs $\mathcal{AF}\subset \mathcal{H}$ defined by the condition that $\lVert B \rVert_g^2<2.$  These results should be compared to the analogous study of the behavior of the Hausdorff dimension of the limit set defined on the space of quasi-Fuchsian representations, see \cite{BT08}, \cite{BRI10} and \cite{McM08}.  Indeed, general theory identifies the function $E$ as recording the Hausdorff dimension of the Gromov boundary $\partial_{\infty}(\widetilde{\Sigma})$ computed in the Gromov metric.

Next, in $\S$\ref{sec: hyp manifolds} and $\S$\ref{sec: limit set} we turn to a study of the interaction between the intrinsic dynamics induced on $\Sigma$ by a pair $(g,B)\in\mathcal{H}$ with the dynamics of a hyperbolic 3-manifold $M$ in which $(\Sigma,g)$ appears as an isometrically immersed minimal surface with second fundamental form $B.$  The central result is a new lower bound on the Hausdorff dimension of the limit set of a quasi-Fuchsian group:
\begin{theorem*}
Let $\Gamma$ be a quasi-Fuchsian group and $\Sigma \rightarrow \Gamma\backslash \mathbb{H}^3$ a $\pi_1$-injective minimal surface with induced metric $g$ and second fundamental form $B.$  Then,
\begin{align}
\frac{1}{\text{\textnormal{Vol}(g)}}\int_{\Sigma}\sqrt{1+\frac{1}{2}\lVert B\rVert_{g}^2}\ dV_{g}\leq \text{\textnormal{H}}.\text{\textnormal{dim}}(\Lambda_{\Gamma}),
\end{align}
where $dV_g$ is the volume form of $g$ and $ \text{\textnormal{H}}.\text{\textnormal{dim}}(\Lambda_{\Gamma})$ is the Hausdorff dimension of the limit set of $\Gamma.$  Furthermore, equality holds if and only if $\Gamma$ is Fuchsian.
\end{theorem*}
Setting $ \text{\textnormal{H}}.\text{\textnormal{dim}}(\Lambda_{\Gamma})$=1 in the above theorem gives a new proof of Bowen's theorem on the Hausdorff dimension of quasi-circles \cite{BOW79}.
\begin{corollary*}
A quasi-Fuchsian group $\Gamma$ is Fuchsian if and only if 
$ \text{\textnormal{H}}.\text{\textnormal{dim}}(\Lambda_{\Gamma})=1.$
\end{corollary*}
Furthermore, the techniques we use allow a quantification of Bowen's theorem which states that if the Hausdorff dimension of the limit set of a quasi-Fuchsian group is very near one, under the additional hypothesis that the quasi-Fuchsian group admits a stable, incompressible minimal surface with a lower bound on the injectivity radius, then the group in question is very close to being Fuchsian.  In particular, we obtain estimates on geometric quantities of the quotient $3$-manifold such as the convex core diameter and volume, the Teichm\"{u}ller distance between the conformal boundaries, and the principal curvatures of the stable minimal surface.  All of these quantities go to zero in a controlled way as the Hausdorff dimension of the limit set goes to $1.$  In particular, we obtain the conclusion that in this setting the stable, incompressible minimal surface in unique.

In $\S$\ref{sec: actions} we further explore the relationship between $\mathcal{H}$ and the space of quasi-Fuchsian representations.  In particular, it has been known since Uhlenbeck \cite{UHL83} that the subset $\mathcal{AF}\subset \mathcal{H}$ of almost-Fuchsian germs defined by the condition $\lVert B\rVert_g^2 < 2$ corresponds to an open subset of quasi-Fuchsian representations: on this subset we show that the action $(g,B)\rightarrow (g,-B)$ corresponds to switching the conformal boundary components in the Bers' simultaneous uniformation parameterization of those quasi-Fuchsian manifolds.  We mention that the failure of this action to globally extend in this fashion comes from critical points of a map,
\begin{align}
\Phi: \mathcal{H}\rightarrow \mathcal{R}(\pi_1(\Sigma),\text{PSL}(2,\mathbb{C})).
\end{align}
These critical points were characterized by Taubes \cite{TAU04} to correspond to minimal surfaces admitting non-trivial deformations which preserve area to second order, known as Jacobi fields.  Hence, the failure of geometric phenomena to transfer between $\mathcal{H}$ and $\mathcal{R}(\pi_1(\Sigma),\text{PSL}(2,\mathbb{C}))$ is governed by the same mechanism which leads to bifurcations in the minimal surface problem; this fascinating behavior is not well understood and only a few cases have been analyzed, we mention \cite{BT84} and \cite{JL95}.  We also observe, as a simple consequence of Sard's theorem, that the set of representations $\rho:\pi_1(\Sigma)\rightarrow \text{PSL}(2,\mathbb{C})$ admitting equivariant minimal immersions of a disk which have a non-trivial equivariant Jacobi field are measure zero in the character variety $\mathcal{R}(\pi_1(\Sigma),\text{Psl}(2,\mathbb{C})).$  This raises basic questions concerning the structure of the critical values of $\Phi.$ How singular are they as subsets of the character variety? What are the components of the complement of all critical values?

Finally, in $\S$\ref{sec: CAT(-1)} we extend the results in $\S$\ref{sec: limit set} from quasi-Fuchsian groups acting on $\mathbb{H}^3$ to convex-cocompact surface groups acting isometrically on any $n$-dimensional CAT$(-1)$ Riemannian manifold.  The generalization of Bowen's theorem in this setting is due to Bonk and Kleiner \cite{BK04}, whose theorem is considerably more general than ours as it applies to general quasi-convex, cocompact group actions on any CAT$(-1)$ metric space.  That being said, our approach also yields a new lower bound on the Hausdorff dimension of the limit set in the $n$-dimensional CAT$(-1)$ Riemannian setting; our methods are less general but provide more information.  Perhaps the work of Mese on minimal surfaces in Alexandrov spaces \cite{MES01} could be utilized to extend the techniques here to surface group actions on more general metric spaces, although we have not considered that here.  In the course of this section, we show that given a convex-cocompact representation $\rho:\pi_1(\Sigma)\rightarrow \text{Isom}(X)$ into the isometry group of a CAT$(-1)$ Riemannian manifold $X,$ any equivariant branched minimal immersion is without ramified branch points; in particular no such branched immersion has false branched points where the immersion is a branched covering onto its image.  This is a straightforward application of a previous result of Gulliver-Tomi \cite{GT89}.
\subsection{Plan of paper}
The first section $\S$\ref{sec: prelim} contains preliminary information regarding minimal surfaces and the geometry and dynamics of groups acting on CAT$(-1)$ spaces.  In $\S$\ref{sec: entropy} we study the dynamically defined entropy function on the space $\mathcal{H}$ of minimal hyperbolic germs and prove the aforementioned properties.  Next, $\S$\ref{sec: hyp manifolds} begins the investigation of the relationship between minimal surfaces and the geometry of hyperbolic 3-manifolds, this culminates in $\S$\ref{sec: limit set} with a new proof of Bowen's theorem on Hausdorff dimension of quasi-circles.  In $\S$\ref{sec: actions}, we prove that the action induced by switching conformal boundary components in the simultaneous uniformization parameterization of quasi-Fuchsian space can be seen, at least for some special quasi-Fuchsian manifolds, in terms of an associated $\frac{\mathbb{Z}}{2\mathbb{Z}}$-action on minimal hyperbolic germs.  Lastly, in $\S$\ref{sec: CAT(-1)} we provide the generalization of the results in $\S$\ref{sec: limit set} to the $n$-dimensional CAT$(-1)$ Riemannian setting.  In the appendix, various technical results needs throughout the paper are collected.

\subsection{Acknowledgments}  
This paper grew out of part of the author's Ph.D thesis completed under the supervision of Dr. William Goldman.  The results came to be after of years of Dr. Goldman's suggestions that I look into the thermodynamic formalism approach to understanding geodesic flows, I should have listened earlier!  I am also grateful to Michelle Lee and Son Lam Ho with whom I ran a small reading seminar in Fall 2012 at the University of Maryland learning the basics of Patterson-Sullivan theory.

\section{Preliminaries}\label{sec: prelim}

Throughout, $\Sigma$ denotes a smooth closed, connected, oriented surface of genus greater than 1.  The universal cover of $\Sigma$ is denoted $\widetilde{\Sigma}.$  Whenever $\Sigma$ is endowed with a Riemannian metric, we equip $\widetilde{\Sigma}$ with the pull-back metric so that the covering projection is a local isometry.

Hyperbolic 3-space $\mathbb{H}^3$ is the unique 3-dimensional, 1-connected, complete Riemannian manifold of sectional curvature $-1.$  Given an immersion,
\[f:\widetilde{\Sigma}\rightarrow \mathbb{H}^3,\]
let $N$ be a locally defined unit normal vector field.  Given coordinate vector fields $\partial_1$ and $\partial_2,$ the \textit{second fundamental form} of the immersion $f$ is the symmetric 2-tensor on $\widetilde{\Sigma}$ defined by,
\[B_{ij}=B(\partial_i,\partial_j)=\langle \nabla_{df(\partial_i)}N,df(\partial_j) \rangle,\]
where $\nabla$ denotes the Levi-Civita covariant derivative associated to the Riemannian metric $\langle -,-\rangle$ on $\mathbb{H}^3.$  As $B$ is a symmetric, covariant 2-tensor, at each point it has a pair of real eigenvalues $\lambda_1$ and $\lambda_2,$ these are the \textit{principal curvatures} of the immersion.  The \textit{mean curvature} H is the trace of $B,$ in terms of the principal curvatures,
\begin{align}
H=\lambda_1+\lambda_2.
\end{align}

If $g$ denotes the pullback of the metric on $\mathbb{H}^3$ via the immersion $f,$ then the equations of Gauss and Codazzi relate $g$ and $B$ via:
\begin{align}
K_g=-1+\lambda_1\lambda_2 \label{eqn: gauss}, \\
(\nabla_{\partial_i}B)_{jk}-(\nabla_{\partial_j}B)_{ik}=0. \label{eqn: codazzi}
\end{align}
Above, $K_g$ denotes the sectional curvature of the metric $g$ and $\nabla$ its Levi-Civita covariant derivative.

The immersion $f:\widetilde{\Sigma}\rightarrow \mathbb{H}^3$ is a \textit{minimal surface} if the mean curvature vanishes identically $H=0.$  This condition is equivalent to the immersion $f$ being critical for area with respect to all compactly supported variations.  The following existence theorem combines the work of Meeks-Simon-Yau \cite{MSY82} and Gulliver \cite{GUL77} (see also \cite{SU82} and \cite{FHS83}).
\begin{theorem}\label{thm: exist minimal}
Let $(M,g)$ be a compact 3-dimensional Riemannian manifold with $\pi_2(M)=0.$  If $M$ has boundary, then assume that $\partial M$ is mean convex.  Then given $f:\Sigma\rightarrow M$ such that $f_{*}:\pi_1(\Sigma)\rightarrow \pi_1(M)$ is injective, there exists an area minimizing immersion $g:\Sigma\rightarrow M$ in the homotopy class of $f.$
\end{theorem}

\textbf{Remark:}  The condition that $\partial M$ is mean convex means that any deformation of $\partial M$ in the outward normal direction is area non-decreasing.  The hypothesis also allows that the boundary be non-smooth, provided it satisfies a natural convexity condition (see \cite{MY82} for details).  In particular, boundaries of open convex sets satisfy this condition.

Now we can introduce the space of minimal hyperbolic germs constructed by Taubes.

Let $(g,B)\in \Gamma(S^2_{>0} T^*\Sigma)\oplus \Gamma(S^2 T^*\Sigma)$ be a pair consisting of a Riemannian metric and symmetric 2-tensor on $\Sigma.$  Such a pair is called a \textit{minimal hyperbolic germ} if $B$ is traceless with respect to $g$ and the Gauss-Codazzi equations \eqref{eqn: gauss} and \eqref{eqn: codazzi} are satisfied.  Letting $\text{Diff}_0(\Sigma)$ be the space of orientation preserving diffeomophisms of $\Sigma$ isotopic to the identity, the space $\mathcal{H}$ of minimal hyperbolic germs is the quotient,
\[\mathcal{H}=\{\text{minimal hyperbolic germs}\}/\text{Diff}_0({\Sigma}),\]
with $\text{Diff}_0(\Sigma)$ acting by pullback on the pair of tensors $(g,B).$  By abuse of notation, we shall say a pair
$(g,B)\in \mathcal{H}$ to mean that the orbit of the pair belongs to $\mathcal{H}.$  The following fundamental theorem of surface theory \cite{UHL83} shows that every element $(g,B)\in\mathcal{H}$ can be integrated to an immersed minimal disk in $\mathbb{H}^3$
with first and second fundamental form $(g,B).$
\begin{theorem} \label{thm: fundamental}
Let $(g,B)\in\mathcal{H}.$  Then there exists an immersion $f:\widetilde{\Sigma}\rightarrow \mathbb{H}^3$ whose induced metric and second fundamental form coincide with the lifts of $g$ and $B$ to $\widetilde{\Sigma}.$  Furthermore, if $O\in\mathbb{H}^3$ is chosen along with a preferred orthonormal frame $\{E_1,E_2,N\}\subset T_{O}\mathbb{H}^3,$ then the map $f$ is uniquely determined by fixing $p\in\widetilde{\Sigma}$ and an orthonormal frame $\{\partial_1,\partial_2\}\subset T_{p}\widetilde{\Sigma}$ and requiring that,
\begin{itemize}
\item $f(p)=O,$
\item $df(\partial_i)=E_i.$
\end{itemize}
\end{theorem}

The following fundamental theorem is due to Taubes,
\begin{theorem}[\cite{TAU04}]
The space of minimal hyperbolic germs $\mathcal{H}$ is a smooth, oriented manifold of dimension $12g-12$ where $g$ is
the genus of $\Sigma.$
\end{theorem}

The Teichm\"{u}ller space $\mathcal{T}$ is the space of isotopy classes of complex structures agreeing with the orientation of $\Sigma,$ which by the K\"{o}ebe uniformization theorem can also be described as the space of isotopy classes of Riemannian metrics of constant curvature $-1.$  For the sake of context, the space of isotopy classes of metrics of constant curvature $-1$ will be called the Fuchsian space, denoted by $\mathcal{F}.$  As such, the Fuchsian space includes into $\mathcal{H}$ via the map,
\begin{align}
\mathcal{F}&\longrightarrow \mathcal{H}\\
g&\mapsto (g,0).
\end{align}

Given a complex structure $\sigma\in\mathcal{T},$ Kodaira-Spencer deformation theory identifies the fiber of the holomorphic cotangent bundle over $\sigma$ as the space of holomorphic quadratic differential $\alpha=\alpha(z) dz^2$ on the Riemann surface $(\Sigma,\sigma).$

The space $\mathcal{H}$ admits an important map to $T^{*}\mathcal{T}$: given $(g,B)\in\mathcal{H},$ let $[g]\in\mathcal{T}$ denote the conformal structure induced by the Riemannian metric $g.$  If $(x_1,x_2)$ are local, isothermal coordinates for the metric $g,$ Hopf observed in \cite{HOP54} that the Codazzi equations along with the fact that $B$ is trace-free imply that the expression:
\begin{align}
\alpha(g,B)=(B_{11}-iB_{12})(x_1,x_2)dz^2,
\end{align}
defines a holomorphic quadratic differential on $(\Sigma,[g])$ where $z=x_1+ix_2.$
This assignment defines a smooth mapping,
\begin{align}\label{map: Psi}
\Psi:\mathcal{H}&\longrightarrow T^{*}\mathcal{T} \\
(g,B)&\mapsto ([g],\alpha).
\end{align}
Furthermore, $\text{Re}(\alpha)=B.$  The obvious action of the circle on $T^{*}\mathcal{T}$ induces an action of $\mathbb{S}^1$ on $\mathcal{H}$ making the above mapping equivariant, under this action the metric $g$ is left completely unchanged.  Hence, the $\mathbb{S}^1$-orbit of minimal surfaces are all mutually isometric, it is often called the \textit{associated family} corresponding to any particular element of the orbit.  Furthermore, the action is free if and only if $B\neq 0.$

\subsection{Almost-Fuchsian germs}

A minimal germ $(g,B)\in\mathcal{H}$ is called almost-Fuchsian if $\lVert B\rVert_g^2<2.$  We denote the set of almost-Fuchsian germs by $\mathcal{AF}.$  These minimal germs directly correspond to hyperbolic 3-manifolds via the following theorem due to Uhlenbeck.

\begin{theorem}[\cite{UHL83}]\label{thm: uhl}
Let $(g,B)\in\mathcal{AF}.$  Then the metric,
\begin{align}
G(g,B)=dt^2+g\left(\cosh(t)\mathbb{I}(-)+\sinh(t)\mathbb{S}(-),\cosh(t)\mathbb{I}(-)+\sinh(t)\mathbb{S}(-)\right)
\end{align}
is a complete hyperbolic metric on $\Sigma\times\mathbb{R}$ where $\mathbb{S}$ is the $(1,1)$ tensor associated to $B$ by raising an index and $\mathbb{I}$ is the identity.  Furthermore,
\begin{itemize}
\item G(g,B) is quasi-isometric to a Fuchsian metric.  Hence, it is quasi-Fuchsian.
\item The slice $\Sigma\times \{0\}$ is an embedded, least area minimal surface in $\Sigma\times \mathbb{R}$ with induced metric and second fundamental form $(g,B).$  This is the only closed minimal surface of any kind in $(\Sigma\times\mathbb{R},G).$
\end{itemize}
\end{theorem}

\subsection{Limit sets of discrete groups and CAT(-1) spaces} \label{sec: cat}

In this section we will introduce the necessary ingredients we will need from the theory developed by Patterson \cite{PAT76}, Sullivan \cite{SUL84}, Bourdon \cite{BOU95} and Coornaert \cite{COO93}.

Let $(X,d)$ be a proper, CAT$(-1)$ metric space (for our needs we may assume this to be a $1$-connected. complete Riemannian manifold of sectional curvature $\leq -1).$  Given $p\in X,$ the geometric (or visual) boundary $\partial_{p,\infty}(X)$ of $X$ is the space of equivalence classes of geodesic rays based at $X.$  Two rays $\gamma,\eta:[0,\infty)\rightarrow X$ are equivalent if there exists $K>0$ such that $d(\gamma(t),\eta(t))<K$ for all $t.$  The \textit{Gromov product} at p is defined by,
\begin{align}
(x,y)_{p}=\frac{1}{2}(d(x,p)+d(y,p)-d(x,y)).
\end{align}
In a CAT$(-1)$ space, this product extends to the geometric boundary via,
\begin{align}
(\eta, \gamma)_{p}=\displaystyle\lim_{t\rightarrow\infty}(\eta(t),\gamma(t))_{p},
\end{align}
for $\eta, \gamma\in \partial_{p,\infty}(X).$
Using the Gromov product, we define the Gromov metric on the geometric boundary by,
\begin{align}
d_p(\eta,\gamma)= \left\{
     \begin{array}{lr}
       e^{-(\eta,\gamma)_{p}} &: \eta\neq \gamma\\
       0 &:  \text{else}
     \end{array}
   \right.
\end{align}
As $p\in X$ varies, the Gromov metrics are all bi-Lipschitz equivalent to one another.

Suppose $\Gamma<\text{Isom}(X)$ is a discrete, convex-cocompact subgroup, this means there is a geodesically convex, $\Gamma-$invariant subset of X upon which $\Gamma$ acts cocompactly.  Define the orbit counting function associated to $\Gamma$ by,
\begin{align}
N_{\Gamma}(x,R)=\lvert\{\gamma\in\Gamma \ \vert \ d(x,\gamma(x))<R\}\rvert.
\end{align}
Then the volume entropy of $\Gamma$ is:
\begin{align}
\delta(\Gamma)=\displaystyle\lim_{R\rightarrow\infty}\frac{\log(N_{\Gamma}(R,x))}{R}.
\end{align}
This number is independent of $x\in X$ and is a measure of the dynamical complexity of the group $\Gamma.$
The \textit{limit set} $\Lambda_{\Gamma}$ of $\Gamma$ is the set of accumulation points of $\Gamma$-orbits of a fixed point $x\in X$ in $\partial_{\infty}(X).$  Equivalently, $\Lambda_{\Gamma}$ is the smallest non-empty, closed $\Gamma$-invariant subset of $\partial_{\infty}(X).$  We will use the following theorem (see \cite{COO93}).
\begin{theorem}\label{thm: enthdim}
Let $\Gamma<\text{Isom}(X)$ be a discrete, convex-cocompact subgroup of isometries of a proper, CAT$(-1)$ metric space $X.$  Then
\[\delta(\Gamma)=\text{\textnormal{H.dim}}(\Lambda_{\Gamma}).\]
Here the Hausdorff dimension is computed using any of the Gromov metrics on $\partial_{\infty}(X).$
\end{theorem}

\section{Entropy} \label{sec: entropy}

In this section, we define the entropy function on the space of minimal hyperbolic germs and discuss some interesting properties.

Given $(g,B)\in\mathcal{H},$ the Gauss equation reads:
\[ K_g=-1-\frac{1}{2}\lVert B\rVert_g^2.\]
Thus, every minimal hyperbolic germ refines the structure of a closed Riemannian surface with sectional curvature bounded above by $-1.$  We remark that by basic comparison geometry, such Riemannian surfaces are proper, locally CAT$(-1)$ metric spaces, and the theory described in section \ref{sec: cat} applies to their universal covers.
\begin{definition}
Given $(g,B)\in\mathcal{H},$ let $B_{g}(p,R)$ be the metric ball in the universal cover $\widetilde{\Sigma}$ of radius $R$ centered at a basepoint $p\in \widetilde{\Sigma}.$  Define the volume entropy as the quantity,
\[E(g,B)=\displaystyle\limsup_{R\rightarrow\infty}\frac{\log\lvert B_g(p,r)\rvert}{R},\]
where $\lvert B_g(p,R)\rvert$ is the Riemannian volume of the ball centered at $p$ of radius $R.$
\end{definition}
Manning introduced this quantity in \cite{MAN79} and showed the limit exists and is independent of basepoint.  Furthermore, in the case where the manifold in question has negative curvature, Manning showed that this quantity equals the topological entropy of the geodesic flow defined on the unit tangent bundle.  Katok, Kneiper and Weiss \cite{KKW91} show that given a $C^{\infty}$-perturbation of a metric of negative curvature, the topological entropy of the geodesic flow also varies smoothly.  Hence:
\begin{proposition}
The volume entropy,
\[E:\mathcal{H}\rightarrow \mathbb{R}\]
is a smooth, non-negative function on the space of minimal hyperbolic germs.  This function equals the topological entropy $h_{top}$ of the geodesic flow on the unit tangent bundle.
\end{proposition}
We will simply refer to this function as the \textit{entropy} of the minimal hyperbolic germ.

We begin with an important lower bound on the entropy.
\begin{theorem} \label{thm: entbound}
The entropy satisfies,
\begin{align}
E(g,B)\geq \frac{1}{\text{Vol}(g)}\int_{\Sigma} \sqrt{1+\frac{1}{2}\lVert B\rVert_{g}^2}\ dV_{g},
\end{align}
and $E(g,B)=1$ if and only if $B=0.$  Furthermore, equality is achieved if and only if $E(g,B)=1.$
\end{theorem}
Before beginning the proof, we introduce an estimate of Manning which easily yields the theorem.  Let $(\Sigma, g)$ be a Riemannian surface with strictly negative sectional curvature.  On the unit tangent bundle of $\Sigma,$ the (normalized) \textit{Liouville measure} $m_{L}$ is a probability measure invariant under the geodesic flow.  In a local trivialization, $m_{L}$ is a constant multiple of the product of Riemannian volume on $\Sigma$ with the standard angle measure on the circle giving it total measure $2\pi.$  The measure theoretic entropy of a metric of constant sectional curvature $-1$ with respect to Liouville measure equals $\sqrt{-2\pi \chi(\Sigma)}.$
\begin{theorem}[\cite{MAN81}]\label{thm: manningbound}
Let $(\Sigma,g)$ be a Riemannian surface of negative curvature.  Then
\[\frac{1}{\sqrt{\text{Vol}(g)}}\int_{\Sigma}\sqrt{-K_g}\ dV_{g}\leq h(m_{L})\]
where $h(m_{L})$ is the measure theoretic entropy of the geodesic flow with respect to Liouville measure.
\end{theorem}
\textbf{Remark:} Note that if $K=-1$ in the above formula, the inequality becomes equality: $h(m_L)=\sqrt{\text{Vol}(g)}=\sqrt{-2\pi \chi(\Sigma)}.$

We now give the proof of Theorem \ref{thm: entbound}.
\begin{proof}
Let $(g,B)$ be a minimal hyperbolic germ.  By the Gauss equation,
\[K_g=-1-\frac{1}{2}\lVert B\rVert_{g}^2.\]
Applying theorem \ref{thm: manningbound} and inserting the above expression for $K_g,$ we obtain the inequality,
\begin{align}\label{inq: entbound}
\frac{1}{\sqrt{\text{Vol}(g)}}\int_{\Sigma}\sqrt{1+\frac{1}{2}\lVert B\rVert_{g}^2}\ dV_{g}\leq h(m_{L}),
\end{align}
where $h(m_L)$ is the measure theoretic entropy of the geodesic flow with respect to Liouville measure for the metric $g.$
By the variational principle (see \cite{HK95}),
\[h(m_{L})\leq h_{top}\left(\frac{1}{\text{Vol}(g)}g\right),\]
where $h_{top}\left(\frac{1}{\text{Vol}(g)}g\right)$ is the topological entropy of the geodesic flow for the normalized Riemannian metric $\frac{1}{\text{Vol}(g)}g.$  But, since $g$ has negative curvature,
\begin{align}
E\left(\frac{1}{\text{Vol}(g)}g,B\right)=h_{top}\left(\frac{1}{\text{Vol}(g)}g\right)
\end{align}
by \cite{MAN79}.
Furthermore, the entropy scales via,
\begin{align}
E\left(\frac{1}{\text{Vol}(g)}g,B\right)=\sqrt{\text{Vol}(g)}E(g,B).
\end{align}
Returning to line \eqref{inq: entbound}, the previous lines imply,
\begin{align}
\frac{1}{\sqrt{\text{Vol}(g)}}\int_{\Sigma}\sqrt{1+\frac{1}{2}\lVert B\rVert_{g}^2}\ &dV_{g}\leq h(m_{L}) \\
&\leq h_{top}\left(\frac{1}{\text{Vol}(g)}g\right) \\
&= E\left(\frac{1}{\text{Vol}(g)}g,B\right)\\
&=\sqrt{\text{Vol}(g)}E(g,B).
\end{align}
Dividing by $\sqrt{\text{Vol}(g)}$ proves the inequality asserted in Theorem \ref{thm: entbound}.

If $E(g,B)=1,$ then,
\begin{align}
\frac{1}{\text{Vol}(g)}\int_{\Sigma} \sqrt{1+\frac{1}{2}\lVert B\rVert_{g}^2}\ dV_{g}\leq 1.
\end{align}
This implies $\lVert B \rVert_{g}^2=0.$

For the other direction, if $\lVert B \rVert_{g}^2=0,$ then the surface has constant sectional curvature $-1$ and one may compute directly that $E(g,B)=1.$  In this case the volume of a ball of radius $R$ in the universal cover is asymptotically $e^R.$

For the final statement, we invoke a deep theorem of Katok \cite{KAT82}.  For a closed surface of genus greater than $1,$ equality holds in,
\[h(m_{L})\leq h_{top}\left(\frac{1}{\text{Vol(g)}}g\right)\]
if and only if the metric is constant negative curvature.  This completes the proof.
\end{proof}

Next, we show that critical points of the restriction of entropy to the space of almost-Fuchsian germs occur precisely at the Fuchsian germs.

\begin{theorem} \label{thm: crit}
Consider the restriction of the entropy to the space of almost-Fuchsian hyperbolic germs,
\[E:\mathcal{AF}\rightarrow \mathbb{R}.\]
This function is critical at $(g,B)$ if and only if $B=0,$ hence if and only if the germ is Fuchsian.  In particular, the entropy increases monotonically along rays $(e^{2u_t}h, tB)$ provided $\lVert tB \rVert_{g_t}^2<2.$  Here $h$ is the hyperbolic metric corresponding to the germ $(h,0),$ so $u_0=0.$
\end{theorem}

Before we prove the theorem, we need to introduce a very useful formula due to Katok, Knieper and Weiss \cite{KKW91} for the first variation of the topological entropy of the geodesic flow on a manifold with negative curvature.  Throughout the rest of this section, dots over a function dependent on a single real parameter $t\in\mathbb{R}$ represent successive derivatives with respect to $t.$

\begin{theorem}[\cite{KKW91}]\label{thm: varent}
Let $g_{t}$ be a smooth path of negatively curved Riemannian metrics on a closed manifold $M.$  If $h_{top}(g_t)$ is the topological entropy of the geodesic flow on $T^1(M)$ for the metric $g_t$ then,
\[\frac{d}{dt}h_{top}(g_t)\big\vert_{t=0}=-\frac{h_{top}(g_0)}{2}\int_{T^1(M)} \frac{d}{dt}g_t(v,v)\big\vert_{t=0} d\mu_0.\]
Here, $\mu_0$ is the Bowen-Margulis measure of maximal entropy for the geodesic flow arising from the metric $g_0.$
\end{theorem}

\textbf{Remark:}  Since we will use none of its properties, we will not define the Bowen-Margulis measure.  Details about its properties and construction can be found in \cite{MAR04}, although a considerably easier construction mirroring \cite{PAT76} can be used for negatively curved Riemannian manifolds.

\begin{proof}[Proof of theorem \ref{thm: crit}]
We already know from Theorem \ref{thm: entbound} that the entropy function is critical at Fuchsian hyperbolic germs.  We have two expressions for the sectional curvature of $g_t=e^{2u_t}h,$
\begin{align}
-1-t^2e^{-4u_t}\lVert B\rVert_{h}^2=K_{g_t}=e^{-2u_t}(-\Delta_{h} u_t -1)
\end{align}
where $\Delta_{h}$ is the Laplace-Beltrami operator associated to the metric $h.$  Taking the time derivative and evaluating at $t_0$ reveals,
\begin{align}\label{eqn: firstvargauss}
-\Delta_h \dot{u_t}=e^{2u_{t_0}}\dot{u_{t}}(\lVert t_{0}B\rVert_{g_{t_0}}^2-2)-t_{0}e^{-2u_{t_0}}\lVert B\rVert_{h}^2.
\end{align}
At a maximum $-\Delta_h \dot{u_t}\geq 0$ which implies that the right hand side of \eqref{eqn: firstvargauss} is non-negative.  The hypothesis $\lVert t_{0}B\rVert_{g_{t_0}}^2<2$ implies that $\dot{u_t}\leq 0.$   Futhermore, if $\dot{u_t}=0$ everywhere then equation \eqref{eqn: firstvargauss} implies that $B=0.$
We have shown,
\begin{align}
\frac{d}{dt}g_t=2\dot{u_t}g_t,
\end{align}
is negative definite.
Applying Theorem \ref{thm: varent},
\begin{align}
\frac{d}{dt}E(g_t,tB)|_{t=t_0}=\frac{d}{dt}h_{top}(g_{t})\big\vert_{t=t_0}=-\frac{h_{top}(g_{t_0})}{2}\int_{T^1(M)}2\dot{u_t} d\mu_{t_0}\geq 0
\end{align}
with equality if and only if $t_0=0.$  This completes the proof.
\end{proof}

Now we show that the entropy function yields a metric on Teichm\"{u}ller space $\mathcal{T}$ whose norm is bounded below by the Weil-Petersson norm.  Recall that given a point $\sigma\in\mathcal{T},$ the cotangent space to $\mathcal{T}$ at $\sigma$ is identified, via Kodaira-Spencer deformation theory (see \cite{KOD05}), with the space of holomorphic quadratic differentials on the Riemann surface $(\Sigma,\sigma).$  The uniformization theorem furnishes a unique hyperbolic metric $h_{\sigma}$ in the conformal class of metrics defined by $\sigma.$  Given two holomorphic quadratic differentials $\alpha$ and $\beta,$ the Weil-Petersson Hermitian pairing is defined by,
\begin{align}
\langle \alpha, \beta \rangle_{WP}=\int_{\Sigma} \frac{\alpha\overline{\beta}}{h_{\sigma}}.
\end{align}
This defines a K\"{a}hler metric on the Teichm\"{u}ller space whose geometry has been intensely studied (for a nice survey see \cite{WOL10}).  A number of geometrically defined potential functions for the Weil-Petersson metric have been found, it seems probable, although we have not found a proof, that the entropy function defined here is yet another potential.
Before we prove this theorem, we need to describe a key formula due to Pollicott \cite{POL94} from which the theorem will follow easily.
\begin{theorem}[\cite{POL94}]\label{thm: 2ndvarent}
Let $g_t$ be a smooth path of Riemannian metrics of negative curvature on a closed manifold $M.$  Then,
\begin{align}
\frac{d^2}{dt^2}h_{top}(g_t)\vert_{t=0}\geq h_{top}(g_0)\left(Var\left(\frac{\dot{g}(v,v)}{2}\right)+
2\left(\int_{T^1(M)}\frac{\dot{g}(v,v)}{2} d\mu_0\right)^2\right) +\\
+h_{top}(g_0)\left(-\int_{T^1(M)} \frac{\ddot{g}(v,v)}{2} d\mu_0
+ \frac{1}{4}\int_{T^1(M)}(\dot{g}(v,v))^2 d\mu_0\right).
\end{align}
Here $\mu_0$ is the Bowen-Margulis measure of maximal entropy for the geodesic flow associated to the metric $g_0.$  Further, dots refer to $t$ derivatives evaluated at $t=0.$
\end{theorem}
In the above, formula, we have not defined the term $\left(Var\left(\frac{\dot{g}(v,v)}{2}\right)\right).$  The reader should see \cite{POL94} for details and definitions, for us the only thing we will need is that $Var(0)=0.$

\begin{theorem}
The Hessian of the entropy function defines a metric on the Fuchsian space $\mathcal{F}\subset \mathcal{H}.$  Furthermore, the norm of this metric is bounded below by $2\pi$ times the norm defined by the Weil-Petersson metric.
\end{theorem}

\begin{proof}
By Theorem \ref{thm: crit}, the entropy function attains a minimum along the Fuchsian locus $\mathcal{F},$ thus its Hessian is a well-defined non-negative quadratic form on the tangent space.  Given a holomorphic quadratic differential $\alpha$ on a Riemann surface $(\Sigma, \sigma),$ for small enough $t>0$ we have the almost-Fuchsian germ $(e^{2u_t}h, t\alpha)\in \mathcal{AF}$ where $h$ is the hyperbolic metric uniformizing $(\Sigma,\sigma).$  Recalling \eqref{eqn: firstvargauss},
\begin{align} \label{eqn: firstvargauss2}
-\Delta_h \dot{u_t}=e^{2u_{t_0}}\dot{u_{t}}(2\lVert t_{0}\alpha\rVert_{g_{t_0}}^2-2)-2t_{0}e^{-2u_{t_0}}\lVert \alpha\rVert_{h}^2,
\end{align}
the maximum principle implies that $\dot{u}_t=0$ at $t=0.$  Hence, all terms in Theorem \ref{thm: 2ndvarent} vanish except for the third containing a second derivative.  Differentiating \eqref{eqn: firstvargauss2} again with respect to $t$ and evaluating at $t=0$ yields,
\begin{align}
-\Delta_{h}\ddot{u}_0=-2\ddot{u}_0 - 2\lVert \alpha \rVert_{h}^2.
\end{align}
Integrating with respect to the Riemannian volume form of $h$ shows,
\begin{align}
\int_{\Sigma} \ddot{u}_0\ dV_{h}=-\int_{\Sigma} \lVert \alpha \rVert_{h}^2\ dV_{h}.
\end{align}
Now, letting $g_t=e^{2u_t}h,$ the fact that $\dot{u}_0=0$ implies that,
\begin{align}
\ddot{g}_0=2\ddot{u}_0 h.
\end{align}
Moreover, the Bowen-Margulis measure for the hyperbolic metric $h$ is simply the Liouville measure on the unit tangent bundle $T^{1}\Sigma.$  Thus,
\begin{align}
\int_{T^{1}\Sigma}\frac{\ddot{g}_0(v,v)}{2} \ d\mu_0&=\int_{T^{1}\Sigma} \ddot{u}_{0} h(v,v)\ d\mu_0 \\
&=\int_{T^{1}\Sigma} \ddot{u_0} \ d\mu_0 \\
&=-2\pi \int_{\Sigma}\lVert \alpha \rVert_{h}^2 \ dV_{h} \\
&= -2\pi\lVert \alpha \rVert_{WP}^2.
\end{align}
Thus, Theorem \ref{thm: 2ndvarent} reveals,
\begin{align}
\frac{d^2}{dt^2}E(g_t,tB)|_{t=0}=\frac{d^2}{dt^2}h_{top}(g_t)|_{t=0}\geq 2\pi\lVert \alpha \rVert_{WP}^2
\end{align}
which completes the proof.
\end{proof}

\section{From germs to hyperbolic 3-manifolds} \label{sec: hyp manifolds}

As recorded in Theorem \ref{thm: fundamental}, every $(g,B)\in\mathcal{H}$ can be integrated to an immersed minimal surface in $\mathbb{H}^3$ with induced metric and second fundamental form $(g,B).$  Furthermore, this immersion is unique up to an isometry of $\mathbb{H}^3.$  Since the data arises from tensors on a closed surface $\Sigma,$ this minimal immersion is equivariant for a representation $\rho:\pi_1(\Sigma)\rightarrow \text{Isom}^{+}(\mathbb{H}^3).$

Given $(g,B)\in\mathcal{H}$ we first describe how to obtain the representation $\rho:\pi_1(\Sigma)\rightarrow \text{Isom}^{+}(\mathbb{H}^3).$  Let
\[f(g,B):\widetilde{\Sigma}\rightarrow \mathbb{H}^3,\]
be an immersion described above and select $\tilde{p}\in\widetilde{\Sigma}$ such that,
\begin{align}
f(\widetilde{p})=O\in\mathbb{H}^3, \\
df(\widetilde{p})(E_i)=F_i\in T_{O}\mathbb{H}^3,
\end{align}
where $\{E_i\}$ constitute an orthonormal frame at $\widetilde{p}\in\widetilde{\Sigma}$ and $\{F_i\}$ is an orthonormal frame at $O\in\mathbb{H}^3.$  Now let $\gamma\in\pi_1(\Sigma).$  Then $f\circ \gamma$ defines a new immersion also with induced metric and second fundamental form $(g,B).$  Thus, by Theorem \ref{thm: fundamental} there exists a unique $\rho(\gamma)\in\text{Isom}^{+}(\mathbb{H}^3)$ such that,
\begin{align}
f(\gamma(\widetilde{p}))=\rho(\gamma)f(\widetilde{p}), \\
df\circ d\gamma(X)=d(\rho(\gamma))\circ df(X),
\end{align}
for all $X\in T_{\widetilde{p}}\widetilde{\Sigma}.$
This assignment defines a map,
\begin{align}
\label{synthmaptorep}
\Phi:\mathcal{H}\longrightarrow \mathcal{R}(\pi_1(\Sigma),\text{Isom}^{+}(\mathbb{H}^3)),
\end{align}
where $\mathcal{R}(\pi_1(\Sigma),\text{Isom}^{+}(\mathbb{H}^3))$ is the space of conjugacy classes of representations of
$\pi_1(\Sigma)$ into $\text{Isom}^{+}(\mathbb{H}^3).$
Note that $\Phi$ is well defined since changing a pair $(g,B)$ by a diffeomorphism isotopic to the identity produces a conjugate representation.
Taubes proved:
\begin{theorem}[\cite{TAU04}]\label{thm: smoothmap}
The image of $\Phi$ consists solely of irreducible representations.
\end{theorem}
As a result of the discussion in the following section, the above theorem shows that the map $\Phi$ takes values in a smooth manifold.

\section{Limit sets of quasi-Fuchsian groups} \label{sec: limit set}
We begin this section with an overview of discrete subgroups acting isometrically on hyperbolic space.  A Kleinian group is a discrete (torsion-free) subgroup $\Gamma<\mathrm{Isom}^{+}(\mathbb{H}^3)\simeq \mathrm{PSL}(2,\mathbb{C})$ of orientation-preserving isometries of hyperbolic 3-space.  Given a Kleinian group $\Gamma,$ the action on $\mathbb{H}^3$ extends to an action on the conformal boundary $\partial_{\infty} (\mathbb{H}^3)\simeq \mathbb{C}\cup \{\infty\}$ by M\"{o}bius transformations.  This action divides $\partial_{\infty}(\mathbb{H}^3)$ into two disjoint subsets: $\Lambda(\Gamma)$ and $\Omega(\Gamma).$  The \textit{limit set} $\Lambda(\Gamma)$ is defined to be the smallest non-empty, $\Gamma$-invariant closed subset of $\partial_{\infty}(\mathbb{H}^3).$  The \textit{domain of discontinuity} $\partial_{\infty}(\mathbb{H}^3)\backslash \Lambda(\Gamma)=\Omega(\Gamma)$ is the largest open set on which $\Gamma$ acts properly discontinuously.  The quotient $M=\mathbb{H}^3/\Gamma$ is a complete hyperbolic 3-manifold with \textit{holonomy} group $\Gamma.$  

A discrete, faithful representation $\rho:\pi_1(\Sigma)\rightarrow \text{Isom}^+(\mathbb{H}^3)$ is \textit{quasi-Fuchsian} if and only if $\Lambda_{\Gamma}$ is a Jordan curve and $\Omega(\Gamma)$ consists of precisely two invariant, connected, simply-connected components.  The representation $\rho$ is \textit{Fuchsian} is $\Lambda_{\Gamma}$ is a round circle.  The quotient $\Omega(\Gamma)/\Gamma=X^+ \cup \overline{X^-}$ is a disjoint union of two marked Riemann surfaces $(X^+,\overline{X^-}),$ each diffeomorphic to $\Sigma,$ where the bar over $X^-$ denotes the surface with the opposite orientation.  The marking, which is a choice of homotopy equivalence $f^{\pm}:\Sigma\rightarrow X^{\pm},$ is determined by the requirement that $f_{*}^{\pm}=\rho.$  Conversely, we have the Bers' simultaneous uniformization theorem,

\begin{theorem}[ \cite{BER60}]\label{thm: bers}
Given an ordered pair of marked Riemann surfaces $(X^+,\overline{X^-})$ each diffeomorphic to $\Sigma,$  there exists an isomorphism $\rho:\pi_1(\Sigma)\rightarrow \Gamma$ onto a quasi-Fuchsian group $\Gamma,$ unique up to conjugation in $\mathrm{PSL}(2,\mathbb{C}),$ such that $\Omega(\Gamma)/\Gamma=X^+\cup \overline{X^-}.$
\end{theorem}
The space of all conjugacy classes of representations of $\pi_1(\Sigma)$ into $\text{Isom}^{+}(\mathbb{H}^3)$ is denoted,
\begin{align}
\mathcal{R}(\pi_1(\Sigma),\text{Isom}^{+}(\mathbb{H}^3)).
\end{align}
For the details concerning the following discussion see \cite{GOL04}.
Via the identification $\text{PSL}_2(\mathbb{C})\simeq \text{Isom}^{+}(\mathbb{H}^3),$ the set of homomorphisms,
\begin{align}
\text{Hom}(\pi_1(\Sigma),\text{Isom}^{+}(\mathbb{H}^3)),
\end{align}
has the structure of an affine algebraic set.  The set of irreducible representations is comprised of those
$\rho\in \text{Hom}(\pi_1(\Sigma),\text{Isom}^{+}(\mathbb{H}^3))$ which do not fix a point in $\partial_{\infty}(\mathbb{H}^3).$  The set of irreducible representations,
\begin{align}
\text{Hom}^{irr}(\pi_1(\Sigma),\text{Isom}^{+}(\mathbb{H}^3)),
\end{align}
is a complex manifold of complex dimension $-3\chi(\Sigma)+3$ upon which the (algebraic) action of $\text{Isom}^{+}(\mathbb{H}^3)$ by conjugation is free and proper.  The quotient,
\begin{align}
 \text{Hom}^{irr}(\pi_1(\Sigma),\text{Isom}^{+}(\mathbb{H}^3))/\text{Isom}^{+}(\mathbb{H}^3)\subset \mathcal{R}(\pi_1(\Sigma),\text{Isom}^{+}(\mathbb{H}^3)),
 \end{align}
 is a complex manifold of complex dimension $-3\chi(\Sigma).$  Quasi-Fuchsian space $\mathcal{QF},$ which consists of conjugacy classes of all quasi-Fuchsian representations, lies in the subspace of irreducible representations as an open subset,
 \begin{align}
 \mathcal{QF}\subset  \text{Hom}^{irr}(\pi_1(\Sigma),\text{Isom}^{+}(\mathbb{H}^3))/\text{Isom}^{+}(\mathbb{H}^3),
 \end{align}
 and thus inherits a complex structure.  With respect to this complex structure, the bijection provided by theorem \ref{thm: bers} becomes a biholomorphism,
\begin{align}
\mathcal{QF}\simeq \mathcal{T}\times \overline{\mathcal{T}}.
\end{align}
The complex structure on $\mathcal{T}$ is the one arising from Kodaira-Spencer deformation theory (see \cite{KOD05}).

If $\rho\in\mathcal{QF}$ is a quasi-Fuchsian representation, the fiber $\Phi^{-1}(\rho)$ of the map $\Phi$ from \eqref{synthmaptorep} consists of $\pi_1(\Sigma)$-injective minimal immersions $\Sigma\rightarrow \rho(\pi_1(\Sigma))\backslash \mathbb{H}^3.$  The general existence theorem \ref{thm: exist minimal} guarantees that this set is always non-empty.

We first show the dynamics of a quasi-Fuchsian representation $\rho$ is at least as complicated as the induced dynamics on an invariant minimal surface in $\Phi^{-1}(\rho).$

\begin{theorem}\label{thm: hdimbound}
Let $\rho\in\mathcal{QF}$ be a quasi-Fuchsian representation and $(g,B)\in \Phi^{-1}(\rho).$  Then,
\begin{align}
\frac{1}{\text{Vol(g)}}\int_{\Sigma}\sqrt{1+\frac{1}{2}\lVert B\rVert_{g}^2}\ dV_{g}\leq \text{\textnormal{H}}.\text{\textnormal{dim}}(\Lambda_{\Gamma}),
\end{align}
with equality if and only if $B$ is identically zero which holds if and only if $\rho$ is Fuchsian.
\end{theorem}

\begin{proof}
Given $(g,B)\in \Phi^{-1}(\rho),$ let $\widetilde{\Sigma}\subset \mathbb{H}^3$ be the $\rho(\pi_1(\Sigma))=\Gamma$-invariant minimal disk with induced metric and second fundamental form $(g,B)$ and fix $x\in \widetilde{\Sigma}.$  Given $R>0,$ define,
\begin{align}
\widetilde{N}_{\Gamma}(R):=\{ \gamma\in \Gamma \ | \ d_{g}(x, \gamma(x))\leq R\},
\end{align}
and,
\begin{align}
N_{\Gamma}(R):=\{ \gamma\in \Gamma \ | \ d_{\mathbb{H}^3}(x, \rho(\gamma)(x))\leq R\},
\end{align}
which denote the number of $\Gamma$-orbits within distance $R$ from $x$ with the distance 
computed in the induced metric $g$ and the hyperbolic metric respectively.  Every point at distance $R$ from $x$ in the metric $g$ is distance less than or equal to $R$ in the hyperbolic metric. Thus,
\begin{align}
\widetilde{N}_{\Gamma}(R)\leq N_{\Gamma}(R).
\end{align}

Taking logarithms of each side, dividing by $R,$ and then letting $R\rightarrow \infty,$ the left hand side converges to the entropy $E(g,B).$ Since quasi-Fuchsian groups are convex-cocompact, Theorem \ref{thm: enthdim} implies that the right hand side converges to the Hausdorff dimension of the limit set of $\Gamma$ (in fact, this was proved much earlier by Sullivan \cite{SUL84}).  This proves:
\begin{align}
E(g,B)\leq \text{\textnormal{H}}.\text{\textnormal{dim}}(\Lambda_{\Gamma}).
\end{align}
Applying Theorem \ref{thm: entbound},
\begin{align}
\frac{1}{\text{Vol(g)}}\int_{\Sigma} \sqrt{1+\frac{1}{2}\lVert B\rVert_{g}^2}\ dV_{g}\leq E(g,B)\leq \text{\textnormal{H}}.\text{\textnormal{dim}}(\Lambda_{\Gamma}).
\end{align}
Furthermore, another appeal to Theorem \ref{thm: entbound} shows that the first inequality is an equality if and only if $B=0.$  But, this occurs if and only if the minimal surface $\widetilde{\Sigma}$ is totally geodesic, which exists if and only if $\Gamma$ is Fuchsian.  This completes the proof.
 \end{proof}
As a corollary we obtain a new proof of Bowen's theorem on the Hausdorff dimension of quasi-circles proved in \cite{BOW79}.
\begin{corollary}
A quasi-Fuchsian representation $\rho\in\mathcal{QF}$ is Fuchsian if and only if
$\text{\textnormal{H}}.\text{\textnormal{dim}}(\Lambda_{\Gamma})=1.$
\end{corollary}
\begin{proof}
If $\rho$ is Fuchsian the result is immediate.  Meanwhile, if $\text{\textnormal{H}}.\text{\textnormal{dim}}(\Lambda_{\Gamma})=1,$
then Theorem \ref{thm: hdimbound} forces $B=0$ for any $(g,B)\in\Phi^{-1}(\rho).$  Since there must exist some $(g,B)\in\Phi^{-1}(\rho),$ this implies that $\rho$ leaves invariant a totally geodesic surface; thus $\rho$ is Fuchsian.
\end{proof}

The previous corollary can be quantified using the estimate in Theorem \ref{thm: hdimbound}.  The following result is a quantitative version of Bowen's rigidity for Hausdorff dimension of the limit set.  Using results in \cite{HW11}, \cite{GHW10} we can give bounds on many geometric quantities pertinent in the study of quasi-Fuchsian groups.  First we need a definition.

\begin{definition}
Fix $\varepsilon_0>0.$  Let $\mathcal{QF}_{\varepsilon_0}\subset \mathcal{QF}$ be the set of all quasi-Fuchsian groups such that there exists a $\pi_1$-injective, stable, immersed minimal surface with induced metric $g$ having injectivity radius at least $\varepsilon_0.$
\end{definition}

We remark that every quasi-Fuchsian group $\Gamma$ such that the injectivity radius of the quotient $\mathbb{H}^3/\Gamma$ is greater than $\varepsilon_0$ belongs to the subset $\mathcal{QF}_{\varepsilon_0}.$

The following is the aforementioned quantitative rigidity theorem.

\begin{theorem}\label{quantrigid}
Fix $\varepsilon_0 >0$ and suppose $\Gamma\in \mathcal{QF}_{\varepsilon_0}.$  Then for all $\varepsilon>0$ there exists $\delta=\delta(\varepsilon_0,\varepsilon)$ such that if $\text{H.dim}(\Lambda_{\Gamma})<1+\delta$ and,
\begin{align}
f:\Sigma\rightarrow \mathbb{H}^3/\Gamma,
\end{align}
is any $\pi_1$-injective, stable, immersed minimal surface with induced metric and second fundamental form $(g,B)$ then:
\begin{enumerate}
\item The second fundamental form satisfies $\lVert B\rVert_{g}^2\leq \varepsilon.$  In particular, $\Gamma$ is almost-Fuchsian.
\item $f:\Sigma\rightarrow \mathbb{H}^3/\Gamma$ is the unique closed minimal surface in $\mathbb{H}^3/\Gamma.$  Furthermore, $f$ is an embedding.
\item The diameter and volume of the convex core $\mathcal{C}(\Gamma)$ of $\mathbb{H}^3/\Gamma$ satisfy,
\begin{align}
\text{\textnormal{Diam}}(\mathcal{C}(\Gamma))\leq \frac{1}{2}\log\left(\frac{1+\sqrt{\varepsilon}}{1-\sqrt{\varepsilon}}\right),
\end{align}
and,
\begin{align}
\text{\textnormal{Vol}}(\mathcal{C}(\Gamma))\leq 2\pi\chi(\Sigma)\left(\frac{\sqrt{\varepsilon}}{1-\varepsilon}+\frac{1}{2}\log\left(\frac{1+\sqrt{\varepsilon}}{1-\sqrt{\varepsilon}}\right)\right).
\end{align}
\item The Teichm\"{u}ller distance between the conformal boundary components of $\mathbb{H}^3/\Gamma$ is at most,
\begin{align}
\log\left(\frac{1+\sqrt{\varepsilon}}{1-\sqrt{\varepsilon}}\right).
\end{align}
\item If $h$ is the hyperbolic metric in the conformal class of $g,$ then $\mathbb{H}^3/\Gamma$ is $(1+\varepsilon)$ bi-Lipschitz to the Fuchsian manifold which uniformizes $h.$  
\end{enumerate}
\end{theorem}

\begin{proof}
Item $(2)$ is an immediate corollary of the fact that $\Gamma$ is almost-Fuchsian.  The estimates in $(3)$ and $(4)$ follow immediately from the estimate on $\lVert B\rVert_{g}$ using the estimates obtained by Huang-Wang and Huang-Guo-Gupta in \cite{HW11}, \cite{GHW10}.  Lastly, item $(5)$ follows from a theorem in \cite{UHL83}.  Hence, the only thing to prove is item $(1).$

An immediate application of a theorem of Schoen \cite{SCH83} implies that there exists a uniform constant $C_1>0,$ independent of $(g,B),$ such that the following $C^0$-estimate holds,
\begin{align}\label{C1Bestimate}
\lVert B\rVert_{C^0}^2\leq C_1.
\end{align}
By Theorem \ref{thm: hdimbound},
\begin{align}
\frac{1}{\text{Vol(g)}}\int_{\Sigma}\sqrt{1+\frac{1}{2}\lVert B\rVert_{g}^2}\ dV_{g}\leq \text{\textnormal{H}}.\text{\textnormal{dim}}(\Lambda_{\Gamma}).
\end{align}
Now let $\varepsilon>0,$ we shall find $\delta>0$ to prove the desired estimate in item $(1).$

Consider the subset of $\Sigma,$
\begin{align}
F_{\varepsilon}=\left\{p\in\Sigma \ | \ {\varepsilon}< \lVert B \rVert_{g}^2 \right\}.
\end{align}
Let $E_{\varepsilon}=\Sigma \backslash F_{\varepsilon}.$ Using \eqref{C1Bestimate}, the Gauss equation,
\begin{align}\label{geq}
K_{g}=-1-\frac{1}{2}\lVert B \rVert_{g}^2,
\end{align}
implies that there exists a uniform constant $C_2$ such that $\lvert K_{g}\rvert \leq C_2.$  In the appendix, we show that \eqref{C1Bestimate} implies that if there exists $p\in F_{2\varepsilon},$ then there exists an $R=R(\varepsilon, \varepsilon_0)>0,$ which depends only on $\varepsilon$ and $\varepsilon_0,$ such that $B_{g}(p,R)\subset F_{\varepsilon}.$  Next, bounded sectional curvature in addition to the assumption that the injectivity radius of $(\Sigma, g)$ is at least $\varepsilon_0$ implies local non-collapsing of volume.  Consequently,  there exists a uniform constant $\kappa=\kappa(\varepsilon,\varepsilon_0)>0$ such that whenever $F_{2\varepsilon}$ is non-empty,
\begin{align}
\text{Vol}(F_{\varepsilon})>\text{Vol}(B_{g}(p,R))>\kappa.
\end{align}
Hence,
\begin{align}
 \text{\textnormal{H}}.&\text{\textnormal{dim}}(\Lambda_{\Gamma})\geq\frac{1}{\text{Vol(g)}}\int_{\Sigma}\sqrt{1+\frac{1}{2}\lVert B\rVert_{g}^2}\ dV_{g} \\
&=\frac{1}{\text{Vol(g)}}\left(\int_{F_{\varepsilon}}\sqrt{1+\frac{1}{2}\lVert B\rVert_{g}^2}\ dV_{g}+\int_{E_{\varepsilon}}\sqrt{1+\frac{1}{2}\lVert B\rVert_{g}^2}\ dV_{g}\right) \\
&>\frac{1}{\text{Vol(g)}}\left(\text{Vol}(F_{\varepsilon})\sqrt{1+\frac{1}{2}\varepsilon}+\text{Vol}(E_{\varepsilon})\right) \\
&>\frac{1}{\text{Vol(g)}}\left(\text{Vol}(F_{\varepsilon})\left(1+\frac{\sqrt{\varepsilon}}{2\sqrt{6}}\right)+\text{Vol}(E_{\varepsilon})\right) \\
&>\frac{1}{\text{Vol(g)}}\left(\text{Vol(g)}+\frac{\kappa \sqrt{\varepsilon}}{2\sqrt{6}}\right)
\end{align}
Hence, we obtain the estimate,
\begin{align}\label{quantbound}
 \text{\textnormal{H}}.\text{\textnormal{dim}}(\Lambda_{\Gamma})>1+\frac{\kappa \sqrt{\varepsilon}}{2\sqrt{6}\text{Vol(g)}}.
\end{align}
Observe that the local non-collapsing of volume implies a uniform lower bound for $\text{Vol}(g).$
Thus, let $\delta$ be defined as,
\begin{align}
\delta=\frac{\kappa \sqrt{\varepsilon}}{2\sqrt{6}\text{Vol(g)}},
\end{align}
and assume the limit set satisfies,
\begin{align}
 \text{\textnormal{H}}.\text{\textnormal{dim}}(\Lambda_{\Gamma})<1+\delta.
\end{align}
Then, the estimate \eqref{quantbound} implies,
\begin{align}
1+\delta<  \text{\textnormal{H}}.\text{\textnormal{dim}}(\Lambda_{\Gamma})<1+\delta,
\end{align}
which is a contradiction.  Hence, we conclude that $F_{2\varepsilon}$ is empty and $\lVert B\rVert_{g}^2\leq 2\varepsilon.$
Replacing $\varepsilon$ with $\frac{\varepsilon}{2}$ completes the proof.
\end{proof}

\textbf{Remark:}  The assumption that the injectivity radius of $(\Sigma, g)$ is bounded below seems essential for the argument since it seems very possible to pack a lot of the mass of $\lVert B\rVert_{g}$ into a region of extremely small volume but controlled diameter.  Nonetheless, the veracity of the purely quasi-Fuchsian group statement that small Hausdorff dimension implies almost-Fuchsian is unclear from our approach.

\textbf{Remark:} This result can also be interpreted as a uniqueness result for the solutions to the boundary value problem for a minimal disk in $\mathbb{H}^3$ asymptotic to the very, fractal boundary curve which is the limit set of $\Gamma.$  Results of this type in minimal surface theory often rest on the hypothesis that the boundary curve is rather smooth but has controlled total geodesic curvature.  Here, the Hausdorff dimension being near $1$ seems to be playing the role of a geodesic curvature bound.

A final corollary is the following $L^1$-bound on the norm of the second fundamental form of a $\pi_1$-injective minimal surface in a quasi-Fuchsian manifold.
\begin{corollary}
If $\rho\in\mathcal{QF}$ is a quasi-Fuchsian representation and $(g,B)\in\Phi^{-1}(\rho),$ then
\begin{align}
\frac{1}{\text{\textnormal{Vol}(g)}}\int_{\Sigma}\sqrt{1+\frac{1}{2}\lVert B\rVert_{g}^2} \ dV_{g}<2.
\end{align}
\end{corollary}

\begin{proof}
This follows from Theorem \ref{thm: hdimbound} since $\text{\textnormal{H}}.\text{\textnormal{dim}}(\Lambda_{\Gamma})<2.$
\end{proof}
This is a necessary condition for a minimal hyperbolic germ to arise as a closed minimal surface in a quasi-Fuchsian manifold; it is unknown to what extent this condition is sufficient.

\section{Comparing actions on $\mathcal{QF}$ and $\mathcal{H}.$} \label{sec: actions}
The space $\mathcal{QF}$ of quasi-Fuchsian representations possesses an anti-holomorphic involution $\iota$ which acts by,
\begin{align}
\iota(X,\overline{Y})=(Y,\overline{X})
\end{align}
where $(X,\overline{Y})\in\mathcal{T}\times\overline{\mathcal{T}}\simeq \mathcal{QF}.$  Meanwhile, the space of minimal germs $\mathcal{H}$ also carries an involution given by the restriction of the $\text{U}(1)$-action to multiplication by $-1$ sending a germ $(g,B)$ to $(g,-B).$  The following Theorem shows that these actions are actually intertwined, at least on the almost-Fuchsian germs.
\begin{theorem}
For all $(g,B)\in\mathcal{AF},$
\begin{align}
\Phi((g,-B))=\iota\circ\Phi((g,B)).
\end{align}
where $\Phi$ is the map defined in \ref{synthmaptorep}
\end{theorem}
\begin{proof}
Theorem \ref{thm: uhl} explicitly expresses the almost-Fuchsian metric on $\Sigma\times \mathbb{R}$ corresponding to the germ $(g,B)\in\mathcal{AF}.$  Fock \cite{FOC07} discovered a complex analytic way to express this metric.  Namely, write $g=e^{2u}\lvert dz\rvert^2$ in conformal coordinates and let $\alpha=\Psi((g,B))$ be the quadratic differential whose real part is $B.$  Then,
\begin{align}
G((g,B))=dt^2+e^{2u}\lvert \cosh(t)dz+\sinh(t)e^{-2u}\overline{\alpha}d\overline{z}\rvert^2
\end{align}
 expresses the almost-Fuchsian metric (expanding out this expression, it is equal to the one appearing in Theorem \ref{thm: uhl}).   The Beltrami differential,
\begin{align}
\mu=e^{-2u}\overline{\alpha},
\end{align}
has the property that the metrics,
\begin{align}
\lvert dz\pm \mu d\overline{z}\rvert^2,
\end{align}
furnish conformal metrics on the two components of the domain of discontinuity for the almost-Fuchsian group corresponding to $(g,B).$  Now, sending $(g,B)$ to $(g,-B)$ sends $\alpha$ to $-\alpha$
which changes $\mu$ to $-\mu.$  Hence, the mapping $\iota$ interchanges the conformal structures on the domain of discontinuity.  This completes the proof.
\end{proof}

If the mapping $\Phi$ were an immersion everywhere, then an analytic continuation argument would show that these actions intertwine on the whole quasi-Fuchsian space.  However, as we now explain this is likely not the case and is related to the bifurcation problem in minimal surface theory.

A minimal hyperbolic germ is non-degenerate if the second variation of area has zero nullity.  More precisely, given the minimal immersion $f:\widetilde{\Sigma}\rightarrow \mathbb{H}^3$ corresponding to a minimal hyperbolic germ $(g,B)\in\mathcal{H}$, let $\nu$ be a unit normal vector field to the image of $f$ and let $u\in C^{\infty}(\Sigma).$  Then lifting $u$ periodically to the universal cover $\widetilde{\Sigma},$ take a normal variation $f_t$ of $f$ such that,
\begin{align}
\frac{d}{dt}f_t\vert_{t=0}=u\nu.
\end{align}
Then the well known formula for the second variation of area (see \cite{CM11}) gives,
\begin{align}
\frac{d^2}{dt^2}Area(f_t^{*}G_{\mathbb{H}^3})\vert_{t=0}=\int_{\Sigma}-u\Delta_{g}u-(\lVert B\rVert_g^2-2)u^2 dV_g.
\end{align}
Here $G_{\mathbb{H}^3}$ is the Riemannain metric on $\mathbb{H}^3.$  A minimal hyperbolic germ is non-degenerate if and only if the \textit{Jacobi operator},
\begin{align}
L_{(g,B)}=-\Delta_{g}-(\lVert B\rVert_g^2 -2),
\end{align}
has no non-zero eigenfunctions with eigenvalue 0, that is there are no non-zero solutions to $L_{(g,B)}u=0.$  Solutions to $L_{(g,B)}u=0$ are called \textit{Jacobi fields.}

The following theorem of Taubes shows the special significance of degenerate minimal germs.

\begin{theorem}[\cite{TAU04}] \label{thm: jacobi}
The vector space of Jacobi fields for the operator $L_{(g,B)}$ is in bijection with the kernel of the differential of the map,
\[\Phi:\mathcal{H}\rightarrow \mathcal{R}(\pi_1(\Sigma),\text{PSL}(2,\mathbb{C})),\]
at the germ $(g,B).$
\end{theorem}
An argument using the continuity method in \cite{UHL83} shows that there do exist minimal germs for which the associated Jacobi operator admits nontrivial solutions.
Using the above theorem, an application of Sard's theorem shows that in terms of representations, this phenomena is non-generic.
\begin{theorem}
The set of conjugacy classes of representations $\rho:\pi_1(\Sigma)\rightarrow \text{Isom}^{+}(\mathbb{H}^3)$ which admit a degenerate, equivariant minimal immersion $f:\widetilde{\Sigma}\rightarrow \mathbb{H}^3$ has measure zero in the character variety.
\end{theorem}
\begin{proof}
By Theorem \ref{thm: jacobi}, a representation is a critical value of the map $\Phi$ if and only if it admits a degenerate, equivariant minimal immersion.   By Theorem \ref{thm: smoothmap}, $\Phi$ takes values in the smooth (even complex) manifold of irreducible representations.  Thus, by Sard's theorem the image of $\Phi$ is measure zero in the character variety which proves the theorem.  
\end{proof}

\textbf{Remark:}  Though we do not pursue this here, it is remarked by Uhlenbeck \cite{UHL83} that the map $\Phi$ is real analytic, although this would require a careful proof that the moduli space $\mathcal{H}$ is real analytic; this should be an application of the real analytic implicit function theorem in the banach manifolds which Taubes uses to construct $\mathcal{H}.$  With these technicalities assumed, the critical sets of $\Phi$ should then be analytic sets.  

\section{From $\mathbb{H}^3$ to CAT$(-1)$ Riemannian manifolds}\label{sec: CAT(-1)}

In this section, we will generalize Theorem \ref{thm: hdimbound} and the subsequent corollaries to the setting of a convex-cocompact representations $\rho:\pi_1(\Sigma)\rightarrow \text{Isom}(X)$ where $X$ is a CAT$(-1)$ Riemannian n-manifold and $n\geq 3.$  

\begin{definition}
An $n$-dimensional Riemannian manifold is CAT$(-1)$ if it is $1$-connected and complete with sectional curvature less than or equal to $-1.$
\end{definition}

Let $f:\widetilde{\Sigma}\rightarrow X$ be a minimal immersion of the universal cover of $\Sigma$ into a CAT$(-1)$ Riemannian $n$-manifold $X.$  We pause to review the basic submanifold theory.  Let $\nabla$ denote the Levi-Civita connection of $X$ and $\nabla^{T}$ the component of $\nabla$ tangential to the image of $f.$  Then the second fundamental form is the symmetric 2-tensor with values in the normal bundle given by,
\begin{align}
B(X,Y)=\nabla_{X} Y-\nabla_{X}^{T} Y,
\end{align}
where $X, Y\in \Gamma(f^{*}TX)$ are tangent to the image of $f.$  Thus, $B$ is a vector valued symmetric $2$-tensor, which we write in a local trivialization as $(B_1,...,B_{n-2}).$  Then each $B_i$ is an ordinary real valued $2$-tensor.  Denoting by $g$ the induced metric of the immersion $f,$ the immersion is \textit{minimal} if the trace of the $B_i$ with respect to $g$ simultaneously vanish:
\begin{align}
\text{tr}_g{B_i}=H_i=0.
\end{align}
For each $i,\ H_i$ is the i-th mean curvature of the immersion.  Denoting the Riemannian metric on $X$ by angled brackets $\langle \ ,\ \rangle,$ the Gauss equation reads:
\begin{align}
K_g=\text{Sec}(\partial_1, \partial_2)-\langle B(\partial_1,\partial_2), B(\partial_1,\partial_2)\rangle +\langle B(\partial_1,\partial_1), B(\partial_2,\partial_2)\rangle,
\end{align}
where $K_g$ is the sectional curvature of $g$ and $\text{Sec}(\partial_1,\partial_2)$ is the sectional curvature of the two plane spanned by $\{\partial_1,\partial_2\}$ computed in the Riemannian metric of $X$.
Choosing isothermal coordinates on $\widetilde{\Sigma}$ for the metric $g$ and writing the result with respect to an orthonormal framing of the normal bundle, the minimality of $f$ implies,
\begin{align}
B(\partial_1,\partial_1)=-B(\partial_2,\partial_2).
\end{align}
This verifies that the Gauss equation in this setting is,
\begin{align}
K_{g}=\text{Sec}(\partial_1,\partial_2)-\frac{1}{2}\lVert B\rVert_{g}^2.
\end{align}
Since the sectional curvature of $X$ is bounded above by $-1,$ as in the hyperbolic case the sectional curvature of $g$ is also bounded above by $-1.$

\begin{definition} Let $X$ be an $n$-dimensional CAT$(-1)$ Riemannian manifold.  A discrete, faithful representation $\rho:\pi_1(\Sigma)\rightarrow \text{\textnormal{Isom}}(X)$ with $\rho(\pi_1(\Sigma))=\Gamma$ is convex-cocompact if there is a geodesically convex, $\Gamma$-invariant subset $K\subset X$ upon which the action of $\Gamma$ is cocompact.
\end{definition}

Now let $\rho:\pi_1(\Sigma)\rightarrow \text{\textnormal{Isom}}(X)$ be a convex-cocompact representation.  Given a marked conformal structure $\sigma \in \mathcal{T}$ in the Teichm\"{u}ller space of $\Sigma,$ it is proved by Goldman-Wentworth \cite{GW07} that there exists a $\rho$-equivariant harmonic map,
\begin{align}
f_{\rho}:(\widetilde{\Sigma},\sigma)\rightarrow X.
\end{align}
Formally, $f_{\rho}$ is a minimizer for the Dirichlet energy, 
\begin{align}
E(u):=\frac{1}{2}\int_{F} \lVert du \rVert^2 dV_{g},
\end{align}
where $u:\widetilde{\Sigma}\rightarrow X$ is any smooth $\rho$-equivariant map, $F\subset \widetilde{\Sigma}$ is a fundamental domain for the action of $\pi_1(\Sigma),$ and the integrand is the Hilbert-Schmidt norm of the differential $du$ times the volume element constructed using any Riemannian metric $g$ in the conformal class of $\sigma.$  This is independent of such a choice since the energy with $2$-dimensional domain is conformally invariant.

Furthermore, results of Al'ber and Hartman \cite{HAR67} guarantee that the harmonic map $f_{\rho}$ is unique
unless $f_{\rho}$ maps onto a single geodesic.  In our situation this never occurs: by equivariance, this would imply that the image of $\rho$ consists of translations along a single geodesic.  Hence, $\rho$ maps $\pi_1(\Sigma)$ faithfully onto an abelian group which is impossible.  

Thus, we conclude that for each convex-cocompact $\rho: \pi_1(\Sigma)\rightarrow \text{Isom}(X),$ there exists a unique $\rho$-equivariant harmonic map.  The regularity theory for harmonic maps implies (see Tromba \cite{TRO92} for a careful proof in the $X=\mathbb{H}^2$ case) that this assignment defines a smooth function on Teichm\"{u}ller space, called the \textit{energy functional}:
\begin{align}
\mathcal{E}_{\rho}:\mathcal{T}\rightarrow \mathbb{R}_{\geq 0},
\end{align}
which records the energy of the unique $\rho$-equivariant harmonic map $f_{\rho}:(\widetilde{\Sigma},\sigma)\rightarrow X.$

Also due to Goldman-Wentworth is the following crucial theorem:
\begin{theorem}[\cite{GW07}]\label{thm: proper}
If $\rho:\pi_1(\Sigma)\rightarrow \text{\textnormal{Isom}}(X)$ is convex-cocompact, the the energy functional $\mathcal{E}_{\rho}$ is a proper function on Teichm\"{u}ller space.
\end{theorem}
As remarked in \cite{GW07}, this implies the existence of a critical point (not necessarily unique!) $\sigma_{\rho}\in\mathcal{T}$ for $\mathcal{E}_{\rho}.$  By Sacks-Uhlenbeck \cite{SU82}, $\sigma_{\rho}$ is critical for $\mathcal{E}_{\rho}$ if and only if ,
\begin{align}
f_{\rho}:(\widetilde{\Sigma},\sigma_{\rho})\rightarrow X,
\end{align}
is a branched isometric immersion, which together with harmonicity implies that, 
\begin{align}
f_{\rho}:(\widetilde{\Sigma},\sigma_{\rho})\rightarrow X,
\end{align}
is a \textit{branched minimal immersion}.  Recall that a branched immersion from an oriented surface $\Sigma$ to any manifold $N$ is a $C^1$-mapping $f:\Sigma\rightarrow N$ which is an immersion on the complement of a finite set of points $\{p_i\},$ and such that the differential vanishes at each $p_i.$  The set $\{p_i\}$ is the set of branch points.

In the present set-up, a theorem of Gulliver-Tomi \cite{GT89} allows us to rule out particular types of branch points.  The question of whether branch points can be entirely avoided seems very difficult; in dimension three we can invoke Theorem \ref{thm: exist minimal} to exclude branch points for branched minimal immersions which are minima for the area functional.  

Before we state the theorem, we need a definition.  Let $f:\Sigma\rightarrow N$ be a branched immersion and let $p\in\Sigma$ be a branch point.  Then on every neighborhood of $p,$ the mapping $f$ is $\ell+1$ to one for some non-negative integer $\ell.$   The number $\ell$ is the \textit{order of ramification} of the branch point $p.$  If none of the branch points of $f$ are ramified, we say that $f$ is an unramified branched immersion.
\begin{theorem}[\cite{GT89}]\label{thm: nobranch}
Let $\Sigma$ be a smooth closed, oriented surface and $N$ a Riemannian manifold.  Suppose,
\begin{align}
f:\Sigma\rightarrow N,
\end{align}
is a branched minimal immersion which induces an isomorphism $f_{*}:\pi_1(\Sigma)\rightarrow \pi_1(N).$  Then $f$ is an unramified branched immersion.
\end{theorem}
\textbf{Remark:}  This Theorem is not stated exactly this way in \cite{GT89}.  First, they restrict to surfaces with boundary as they have applications to the Plateau problem in mind.  Nonetheless, the proof they present works, and is substantially simplified, in the closed case.  For completeness, we present this simplified proof, assuming a certain factorization theorem for branched immersions, in an Appendix at the end of this paper.  Second, they prove the theorem more generally for any branched immersion which has the \textit{unique continuation property}, although they note that branched minimal immersions are premiere examples of this phenomena.  The unique continuation property guarantees that the branched immersion is uniquely determined by it's value restricted to any open subset of $\Sigma.$  

We can finally prove:
\begin{theorem}\label{thm: minexist}
Let $X$ be an $n$-dimensional CAT$(-1)$ Riemannian manifold.  Let $\rho:\pi_1(\Sigma)\rightarrow \text{\textnormal{Isom}}(X)$ be a convex-cocompact representation.  Then there exists a $\rho$-equivariant unramified branched minimal immersion.
\begin{align}
f:\widetilde{\Sigma}\rightarrow X.
\end{align}
\end{theorem}

\begin{proof}
By Theorem \ref{thm: proper}, the energy functional $\mathcal{E}_{\rho}:\mathcal{T}\rightarrow \mathbb{R}$ is proper, and by the discussion that follows Theorem \ref{thm: proper} this implies the existence of a $\rho$-equivariant branched minimal immersion,
\begin{align}
f:\widetilde{\Sigma}\rightarrow X.
\end{align}
As $\rho$ is discrete and faithful, the mapping $f$ descends to the quotient as a branched minimal immersion,
\begin{align}
f':\Sigma\rightarrow X/\Gamma,
\end{align}
where $\Gamma=\rho(\pi_1(\Sigma)).$  Since $\rho$ is faithful and $f'_{*}=\rho,$ it follows that $f'$ induces an isomorphism on the level of fundamental group, 
\begin{align}
f'_{*}:\pi_1(\Sigma)\rightarrow \Gamma.
\end{align}
Thus, the hypotheses of Theorem \ref{thm: nobranch} are satisfied which implies that $f'$ is an unramified branched immersion, hence so is $f.$ This completes the proof.  
\end{proof}

With the above discussion in place, the following theorem is a generalization of Theorem \ref{thm: hdimbound}.
\begin{theorem}\label{thm: gendimbound}
Let $X$ be an $n$-dimensional \text{CAT}$(-1)$ Riemannian manifold.  Let $\rho:\pi_1(\Sigma)\rightarrow \text{Isom}(X)$ be a convex-cocompact representation with $\rho(\pi_1(\Sigma))=\Gamma.$  Let $f:\Sigma\rightarrow X/\Gamma$ be a $\pi_1$-injective branched minimal immersion with induced metric $g$ and second fundamental form $B.$  Let $\hat{\Sigma}=\Sigma\backslash \{p_1,...,p_k\}$ where $\{p_1,...,p_k\}$ is the locus of branch points.  Then,
\begin{align}
\frac{1}{\text{Vol(g)}}\int_{\hat{\Sigma}}\sqrt{-\text{Sec}(\partial_1,\partial_2)+\frac{1}{2}\lVert B\rVert_{g}^2}\ dV_{g}\leq \text{\textnormal{H}}.\text{\textnormal{dim}}(\Lambda_{\Gamma}).
\end{align}
\end{theorem}

\textbf{Remark:}  In the above theorem the Hausdorff dimension of $\Lambda_{\Gamma}$ is being computed with respect to the Gromov metric on $\partial_{\infty}(X).$

\begin{proof}
First note that by Theorem \ref{thm: minexist}, there exists a $\pi_1$-injective branched minimal immersion in the quotient $X/\Gamma,$
\begin{align}
f:\Sigma\rightarrow X/\Gamma.
\end{align}
If the dimension of $X$ is equal to $3,$ then Theorem \ref{thm: exist minimal} implies that $f$ can be taken to be an immersion and the proof of the theorem follows exactly as in Theorem \ref{thm: hdimbound}.  So assume the dimension of $X$ is greater than $3.$

The strategy is as follows: let $\mathfrak{B}:=\{p_1,...p_k\}$ be the branching locus.  For any small $\varepsilon>0,$ we will construct an immersion which is equal to $f$ away from $\varepsilon$-neighborhoods of the $p_i$ and whose induced metric has negative sectional curvature.  We then apply the argument of Theorem \ref{thm: hdimbound} to these perturbed surfaces; a simple limiting argument will complete the proof.  We state the exact requirements for the perturbation in the following claim, relegating the tangential proof to the appendix.

\textbf{Claim:}  Let $(M,h)$ be a Riemannian manifold with sectional curvature at most $-1$ and of dimension at least $4,$ and suppose,
\begin{align}
f:\Sigma\rightarrow M,
\end{align}
is a branched minimal immersion.  Then there exists $\varepsilon_0> 0$ and smooth maps,
\begin{align}
f_{\varepsilon}:\Sigma\rightarrow M,
\end{align}
for all $\varepsilon\in [0,\varepsilon_{0}]$ satisfying the following properties:
\begin{enumerate}
\item The maps $f_{\varepsilon}$ are immersions for $\varepsilon>0$ and $f_0=f.$  Denote the induced metrics by 
$f_{\varepsilon}^{*}h=g_{\varepsilon}.$  Also, let $f^{*}h=g$ be the induced metric via $f$ on the complement of the branch points $\mathcal{B}.$
\item \label{imm equal} The maps $f_{\varepsilon}$ satisfy $f_{\varepsilon}=f$ on $\Sigma(\varepsilon)$ where, 
\begin{align}
\Sigma(\varepsilon)=\Sigma\backslash \{ \cup B_{g}(p_i, 3\varepsilon) \},
\end{align}
with the union taken over the set of branch points $\mathcal{B}.$ 
\item Let $K_{\varepsilon}$ be the Gauss curvature of the metric $g_{\varepsilon}$ and $K_g$ be the Gauss curvature of the metric $g$ on the complement of the branch points $\mathcal{B}.$  Then 
$K_{\varepsilon}\rightarrow K_g$ pointwise on the complement of $\mathcal{B}$ in $\Sigma.$  Note that this formally follows from the previous property.
\item  There exists $\kappa(\varepsilon_0)>0$ such that $K_{\varepsilon}(p)<-\kappa(\varepsilon_0)$ for all $\varepsilon\in (0,\varepsilon_{0}]$ and for all $p\in\Sigma.$
\end{enumerate} 

Assuming the claim, the proof of the theorem is as follows.
 
The exact same argument as in Theorem \ref{thm: hdimbound} implies,
\begin{align}\label{epsbound}
\frac{1}{\text{\textnormal{Vol}}(g_{\varepsilon})}\int_{\Sigma} \sqrt{-K_{{\varepsilon}}} dV_{g_{\varepsilon}}\leq E(g_{\varepsilon})\leq \text{\textnormal{H.dim}}(\Lambda_{\Gamma}).
\end{align}
where $E(g_{\varepsilon})$ is the volume entropy of the induced metric $g_{\varepsilon}$ from the immersion $f_{\varepsilon}.$

Now recall the definition,
\begin{align}
\Sigma(\varepsilon)=\Sigma\backslash \{ \cup B_{g}(p_i, 3\varepsilon) \},
\end{align}
where the union is taken over all the branch points.  By \eqref{imm equal}, $g_{\varepsilon}=g$ on $\Sigma(\varepsilon),$ thus, 
\begin{align}
\frac{1}{\text{\textnormal{Vol}}(g_{\varepsilon})}\int_{\Sigma(\varepsilon)}\sqrt{-K_{g}} dV_{g}\leq 
\frac{1}{\text{\textnormal{Vol}}(g_{\varepsilon})}\int_{\Sigma} \sqrt{-K_{{\varepsilon}}} dV_{g_{\varepsilon}}.
\end{align}
Applying \eqref{epsbound} yields,
\begin{align} \label{epsbound2}
\frac{1}{\text{\textnormal{Vol}}(g_{\varepsilon})}\int_{\Sigma}\chi_{\varepsilon}\sqrt{-K_{g}} dV_{g}\leq \text{\textnormal{H.dim}}(\Lambda_{\Gamma}).
\end{align}
where $\chi_{\varepsilon}$ is the characteristic function of $\Sigma(\varepsilon).$

For $\varepsilon$ varying in any compact set including $0,$ the volume satisfies $\text{Vol}(g_{\varepsilon})>C$ for some $C>0$ independent of $\varepsilon.$  Hence, the bound,
\begin{align}
\frac{1}{\text{\textnormal{Vol}}(g_{\varepsilon})}\chi_{\varepsilon}\sqrt{-K_g}\leq \frac{1}{C}\sqrt{-K_g},
\end{align}
is valid on all of $\Sigma.$  Applying the Cauchy-Schwarz inequality reveals,
\begin{align}
\left(\int_{\Sigma}\sqrt{-K_g}dV_{g}\right)^2\leq \text{\textnormal{Vol}}(g)\int_{\Sigma} -K_g dV_{g}.
\end{align}
Furthermore, the Gauss equation implies a uniform upper bound on $\textnormal{Vol}(g).$
Additionally, the Gauss-Bonnet theorem (for surfaces with cone singularities) implies,
\begin{align}
\int_{\Sigma} -K_g dV_{g}< \infty,
\end{align}
which ensures that $\sqrt{-K_g}$ is integrable with respect to $dV_{g}.$  Since,
\begin{align}
\frac{1}{\text{\textnormal{Vol}}(g_{\varepsilon})}\chi_{\varepsilon}\sqrt{-K_g}\rightarrow \frac{1}{\text{\textnormal{Vol}}(g)}\sqrt{-K_g},
\end{align}
pointwise on the complement of the branch points, the dominated convergence theorem implies,
\begin{align}
\frac{1}{\text{\textnormal{Vol}}(g_{\varepsilon})}\int_{\Sigma}\chi_{\varepsilon} \sqrt{-K_{g}} dV_{g}\rightarrow \frac{1}{\text{\textnormal{Vol}}(g)}\int_{\Sigma} \sqrt{-K_g} dV_{g},
\end{align}
as $\varepsilon\rightarrow 0.$
Hence, letting $\varepsilon\rightarrow 0$ in $\eqref{epsbound2}$ yields,
\begin{align}
\frac{1}{\text{\textnormal{Vol}}(g)}\int_{\hat{\Sigma}} \sqrt{-K_{g}} dV_{g}\leq \text{\textnormal{H.dim}}(\Lambda_{\Gamma}).
\end{align}
This completes the proof.
\end{proof}

\begin{corollary}
Let $\rho:\pi_1(\Sigma)\rightarrow \text{\textnormal{Isom}}(X)$ be a convex-cocompact representation with $\rho(\pi_1(\Sigma))=\Gamma.$  Then $\text{\textnormal{H.dim}}(\Lambda_{\Gamma})=1$ if and only if there exists a $\Gamma$-invariant totally geodesic embedding,
\begin{align}
f:\mathbb{H}^2\rightarrow X.
\end{align}
\end{corollary}

\begin{proof}
If there exists a $\rho$-equivariant totally geodesic embedding,
\begin{align}
f:\mathbb{H}^2\rightarrow X,
\end{align}
then $f$ extends to a bi-lipschitz map,
\begin{align}
\overline{f}:\partial_{\infty}(\mathbb{H}^2)\rightarrow \Lambda_{\Gamma},
\end{align}
equipped with their natural (bi-Lipschitz equivalence classes of) Gromov metrics.  Since the Hausdorff dimension of $\partial_{\infty}(\mathbb{H}^2)$ in any of these metrics is $1$ and Hausdorff dimension is a bi-Lipschitz invariant, 
\begin{align}
\text{\textnormal{H.dim}}(\Lambda_{\Gamma})=1.
\end{align}
In the other direction, assume $\text{\textnormal{H.dim}}(\Lambda_{\Gamma})=1.$  Then Theorem \ref{thm: gendimbound} implies the estimate,
\begin{align}
\frac{1}{\text{Vol(g)}}\int_{\hat{\Sigma}}\sqrt{-\text{Sec}(\partial_1,\partial_2)+\frac{1}{2}\lVert B\rVert_{g}^2}\ dV_{g}\leq \text{\textnormal{H}}.\text{\textnormal{dim}}(\Lambda_{\Gamma})=1,
\end{align}
where $(g,B)$ are the induced metric and second fundamental form of any $\pi_1$-injective, branched minimal immersion,
\begin{align}
f:\Sigma \rightarrow X/\Gamma.
\end{align}
We claim that $f$ is actually an immersion.  Suppose to the contrary that $f$ has branch points.

Since $X$ is CAT$(-1),\ \text{\textnormal{Sec}}(\partial_1,\partial_2)\leq -1.$  The only possibility is $\text{\textnormal{Sec}}(\partial_1,\partial_2)=-1$ and $\lVert B \rVert_g=0.$  Then, the Gauss equation becomes,
\begin{align}
K_g= \text{\textnormal{Sec}}(\partial_1,\partial_2)-\frac{1}{2}\lVert B\rVert_g^2=-1,
\end{align}
which is valid away from the branching locus of $f.$
Hence, there is an isometry $\widetilde{f}:\widetilde{\Sigma}\rightarrow \mathbb{H}^2$ which is equivariant for a representation $\widetilde{\rho}:\pi_1(\Sigma)\rightarrow \text{Isom}(\mathbb{H}^2).$  The representation $\widetilde{\rho}$ is the monodromy of a branched hyperbolic structure, thus, by a theorem of Goldman \cite{GOL80} the representation $\widetilde{\rho}$ is not simultaneously discrete and faithful; otherwise it would be the monodromy of an unbranched hyperbolic structure.  But, $\iota\circ \widetilde{\rho}=\rho$ where $\iota$ is the inclusion of $\rho(\pi_1(\Sigma))$ into $\text{Isom}(X).$  This contradicts the fact that $\rho$ is discrete and faithful, hence $f$ has no branch points and it is an immersion.  

Next, pick $p\in\widetilde{\Sigma}$ and consider the $2$-plane $P\subset T_{f(p)}X$ tangent to $f(\widetilde{\Sigma})$ at $f(p),$ where here we use the same name for the lifted map,
\begin{align}
f:\widetilde{\Sigma}\rightarrow X.
\end{align}
Since $X$ has negative sectional curvature and $f$ is totally geodesic, the exponential map from $f(p)$ in the directions spanned by $P$ gives a diffeomorphism between $P$ and $f(\widetilde{\Sigma})$ which proves that $f$ is actually an embedding.  Thus, $f:\widetilde{\Sigma}\rightarrow X$ is a $\rho$-equivariant, totally geodesic embedding of the hyperbolic plane into $X$.  This completes the proof.  
\end{proof}
We close the paper with a series of remarks about the results we have obtained.

\textbf{Remark:}  We could also use the fact that our totally geodesic branched immersion is unramified to prove that it is actually an embedding.  In terms of the proof we give, this follows readily from the fact that the map uniformizing a branched hyperbolic structure on $\widetilde{\Sigma}$ is ramified at the branch points.

\textbf{Remark:} We emphasize that these results give a new proof of rigidity of Hausdorff dimension for convex-cocompact closed surface subgroups of rank-$1$ Lie groups of non-compact type.  This includes quasi-Fuchsian groups in real hyperbolic space $\mathbb{H}^n$ and complex quasi-Fuchsian groups in complex hyperbolic space $\mathbb{CH}^n.$  In addition, the lower bounds we obtain give a geometric explanation for why the Hausdorff dimension of the limit set grows as the lattice of orbits becomes more geometrically distorted in $X.$

\textbf{Remark:}  The applications of the techniques here have not been extinguished: given a discrete, faithful surface group representation into the isometry group of some manifold $X$ for which the associated energy functional on Teichmuller space is proper, one obtains an equivariant unramified branched minimal surface.  For example, if $X=G/K$ is a higher rank symmetric space, the situation is more complicated due to the existence of flats, and we would no longer make a statement about the Hausdorff dimension of the limit set, but rather about the growth of orbits directly.  In any case, there are a wealth of examples (Hitchin representations, maximal representations) of this type due to the work of Lauborie \cite{LAB08}; we plan to study these problems in a future paper.

\textbf{Remark:} Lastly, it is interesting to note that the Hausdorff dimension of the limit set controls the average norm of the second fundamental form of any $\pi_1$-injective, negatively curved surface in the quotient manifold.  It seems likely that this fact can be exploited in other circumstances than those investigated here.

\section{Appendix}

\subsection{Gradient estimate at scale epsilon}
In the proof of the quantitative Bowen rigidity Theorem \ref{quantrigid}, we promised the following proposition.  We refer back to the proof for the notation.
\begin{proposition}
Fix an $\varepsilon>0$ and assume there exists $p\in F_{2\varepsilon}.$  Then there exists $R=R(\varepsilon)>0$ such that $B_{g}(p,R)\subset F_{\varepsilon}.$
\end{proposition}
\begin{proof}
First, recall that by the theorem of Scheon \cite{SCH83}, there exists $C_1>0$ such that,
\begin{align}\label{schoenbound}
\lVert B\rVert_{g}^2<C_1.
\end{align}
Then the Gauss equation,
\begin{align}
K_{g}=-1-\frac{1}{2}\lVert B\rVert_{g}^2,
\end{align}
implies that there exists $C_2>0$ such that $-C_{2}< K_{g}  \leq -1.$  Writing $g$ as a conformal deformation of the hyperbolic metric $h$ in the same conformal class $g=e^{2u}h,$ we may express the Gauss equation with respect to $h$ via,
\begin{align} \label{gaussconformal}
K_{g}=-e^{2u}(\Delta_{h}u+1).
\end{align}
Using the uniform bounds on $K_g$ and applying the maximum and minimum principle to \eqref{gaussconformal} implies there exists $C_3>0$ and the following uniform bound,
\begin{align}
-C_3<u\leq 0.
\end{align}
Hence, the injectivity radius of $(\Sigma, h)$ is at least that of $(\Sigma, g).$  Recall that there exists a holomorphic quadratic differential $\alpha$ such that $\frac{1}{2}\lVert B\rVert_{g}^2=\lVert \alpha \rVert_{g}^2.$  The uniform bound on the conformal factor $u$ in tandem with \eqref{schoenbound} implies that there exists a uniform bound on $\lVert \alpha \rVert_{h}^2.$  Hence, over the hyperbolic surface $(\Sigma, h),$ the set of holomorphic quadratic differentials whose real part appears as the second fundamental form $B$ of a stable, immersed minimal surface is compact.  Since the injectivity radius of $(\Sigma, h)$ is uniformly bounded below, the Mumford compactness theorem implies that the metrics $h$ live in a compact set in the moduli space of all hyperbolic metrics on $\Sigma.$  Hence, we may define,
\begin{align} \label{Rdef}
R:=\min_{(g,B)} \{R>0 \ |\ \lVert B \rVert_{g}(x)>\varepsilon \ \text{for all}\ x\in B_{g}(p,R)\}.
\end{align}
Here, the minimum is taken over all $(g,B)$ first and second fundamental forms of stable, immersed minimal surfaces in quasi-Fuchsian 3-manifolds such that the injectivity radius of $g$ has a uniform lower bound, and points $p\in \Sigma$ such that $p\in F_{2\varepsilon},$ namely that $\lVert B\rVert_{g}(p)> 2\varepsilon.$
  Certainly, at any pair $(g,B)$ which is the first and second fundamental form of a stable, immersed minimal surface in a quasi-Fuchsian $3$-manifold,  there exists such an $R>0$ simply by continuity.  By our previous discussion, such $(g,B)$ vary over a compact set, hence we conclude that the minimum in \eqref{Rdef} is attained and hence $R>0.$  This completes the proof.
\end{proof}

\subsection{Gulliver-Tomi theorem}

In this section of the appendix we provide a simplified proof of the Gulliver-Tomi theorem.  Let $\Sigma$ be a closed, oriented surface of genus greater than one and $M$ any $n$-dimensional manifold for $n\geq 2.$  
\begin{definition}\label{def: branch}
A $C^1$-mapping $f:\Sigma\rightarrow M$ is a branched immersion if there exists a finite set of points $p_1,...p_k\in \Sigma$ such that $f|_{\Sigma\backslash \{p_1,...,p_k\}}$ is an immersion.  Furthermore, for each $i$ there exists positive integers $q_i,$ open sets $U_i\subset \Sigma, V_i\subset M$ containing $p_i, f(p_i)$ respectively, and coordinate charts $\phi^i:U_i \rightarrow \mathbb{C}$ and $\eta^i: V_i\rightarrow \mathbb{R}^n$ such that in these coordinates:
\begin{align}
f^1(x+iy)&=\text{\textnormal{Re}}((x+iy)^{q_i})+o(\lvert x+iy\rvert^{q_i}), \\
f^2(x+iy)&=\text{\textnormal{Im}}((x+iy)^{q_i})+o(\lvert x+iy\rvert^{q_i}), 
\end{align}
\begin{center}
$f^j=o(\lvert x+iy \rvert^{q_i}), \ \frac{\partial f^j}{\partial x}=o(\lvert x+iy\rvert^{{q_i}-1}), \ \frac{\partial f^j}{\partial y}=o(\lvert x+iy\rvert^{{q_i}-1}), \ 3\leq j\leq n.$
\end{center}
The points $p_i$ are the branch points of the immersion $f.$  Each number $q_i -1$ is the order, or index, of the branch point $p_i.$
\end{definition}
Examples of branched immersions include holomorphic maps between Riemann surfaces, or more generally the critical points of energy functionals which we consider in this paper, which are minimal surfaces on the complement of the branch points.

Let $f:\Sigma\rightarrow M$ be a branched immersion and let $p\in\Sigma$ be a branch point.  Then on every neighborhood of $p,$ the mapping $f$ is $\ell+1$ to one for some non-negative integer $\ell.$   The number $\ell$ is the \textit{order of ramification} of the branch point $p.$  If none of the branch points of $f$ are ramified, we say that $f$ is an unramified branched immersion.

A remarkable theory of branched immersions has been developed by many mathematicians, most notably Gulliver, Osserman and Royden \cite{GOR73}.  One of the key elements is a factorization theorem for branched immersions with the unique continuation property.  Firstly recall that $f$ has the unique continuation property is $f$ is uniquely determined by it's value on any open set $U\subset \Sigma.$   

Define an equivalence relation on non-branch points of $\Sigma$ as follows: $x\sim y$ for $x,y\in \Sigma$ if and only if there exists open sets $U,V\subset \Sigma$ containing $x$ and $y$ respectively and an orientation preserving homeomorphism $h: U\rightarrow V$ such that $f|_{U}=f\circ h.$  The following theorem is proved by Gulliver, Osserman and Royden \cite{GOR73}.
\begin{theorem}\label{thm: gullform}
Let $f:\Sigma\rightarrow M$ be a branched immersion with the unique continuation property.  The quotient, 
\begin{align}
\Sigma^{\sim}:=\Sigma/\sim
\end{align}
is a closed, oriented surface, the quotient map $\pi:\Sigma\rightarrow \Sigma^{\sim}$ is a branched covering, and there exists a unique $f^{\sim}:\Sigma^{\sim}\rightarrow M$ such that $f=f^{\sim}\circ \pi.$  Furthermore, the branch points  of $\pi$ coincide with those of $f$ and the order of ramification at each branch point is also equal.
\end{theorem} 
\textbf{Remark:}  If $f:\Sigma\rightarrow \Sigma'$ is a branched covering, then $\Sigma^{\sim}=\Sigma',\ f=\pi,$ and $f^{\sim}=\text{Id.}$

Now, the theorem we wish to prove follows quite rapidly from the Riemann-Hurwitz formula:
\begin{theorem}
Let $f:\Sigma\rightarrow M$ be a branched immersion with the unique continuation property such that $f_{*}:\pi_1(\Sigma)\rightarrow \pi_1(M)$ is an isomorphism.  Then $f:\Sigma\rightarrow M$ has no ramified branch points, thus $f$ is an unramified branched immersion.
\end{theorem}
\begin{proof}
Let ${p_i}\subset \Sigma$ be the branch points of $f$ with order of ramification $\ell_{i} -1.$  Form the ramification divisor 
\begin{align}
D_{f}=\sum (\ell_{i}-1)[p_i].
\end{align}
Then $\text{Deg}(D_{f})=\sum (\ell_i-1)$ and $f$ is unramified if and only if $\text{Deg}(D_{f})=0.$  If $\pi:\Sigma\rightarrow \Sigma^{\sim}$ is the branched covering provided by Theorem \ref{thm: gullform}, then the Riemann-Hurwitz formula implies there exists $N>0$ such that
\begin{align}
\chi(\Sigma)=N\chi(\Sigma^{\sim})-\text{\textnormal{Deg}}(D_{f}).
\end{align}
Since $f$ is an isomorphism on fundamental group and $f=f^{\sim}\circ \pi,$ it follows that $f^{\sim}_{*}:\pi_1(\Sigma^{\sim})\rightarrow \pi_1(M)$ is surjective and $\pi_*:\pi_1(\Sigma)\rightarrow \pi_1(\Sigma^{\sim})$ is injective.  The injectivity of $\pi_{*}$ implies that $\chi(\Sigma^{\sim})<0,$ since a closed surface group of genus greater than one can not surject onto the trivial group or onto $\mathbb{Z}\oplus \mathbb{Z}.$  Thus, the Riemann-Hurwitz formula implies that 
\begin{align}
2-2g&=N(2-2\widetilde{g})-\text{\textnormal{Deg}}(D_{f}) \\
&\leq (2-2\widetilde{g})-\text{\textnormal{Deg}}(D_{f}),
\end{align}
where $\widetilde{g}$ is the genus of $\Sigma^{\sim}.$
This inequality implies that
\begin{align}\label{eqn: pos}
0\leq \text{\textnormal{Deg}}(D_{f})\leq 2(g-\widetilde{g}).
\end{align}
Lastly, since $\pi_1(\Sigma^{\sim})$ surjects onto $\pi_1(M)\simeq \pi_1(\Sigma),$ the minimal cardinality of a generating set of $\pi_1(\Sigma^{\sim}),$ which is $2\widetilde{g},$ is at least as large as the minimal cardinality of a generating set of $\pi_1(\Sigma),$ which is $2g.$ Thus,  
\begin{align}
2\widetilde{g}\geq 2g.
\end{align}
Combining this with \eqref{eqn: pos} implies that $g=\widetilde{g}.$  Hence, \eqref{eqn: pos} implies that $\text{Deg}(D_{f})=0,$ which, as stated previously, implies that $f$ has no ramification points.  This proves the theorem.
\end{proof}

\subsection{Perturbing branched immersions}

In this section of the appendix, we give a complete proof of the claim contained in the proof of Theorem \ref{thm: gendimbound}.  We restate the result here as a proposition.

\begin{proposition}\label{pert branched}
 Let $(M,h)$ be a Riemannian manifold with sectional curvature at most $-1$ and of dimension at least $4,$ and suppose
\begin{align}
f:\Sigma\rightarrow M,
\end{align}
is a branched minimal immersion.  Then there exists $\varepsilon_0> 0$ and smooth maps,
\begin{align}
f_{\varepsilon}:\Sigma\rightarrow M,
\end{align}
for all $\varepsilon\in [0,\varepsilon_{0}]$ satisfying the following properties:
\begin{enumerate}
\item The maps $f_{\varepsilon}$ are immersions for $\varepsilon>0$ and $f_0=f.$  Denote the induced metrics by 
$f_{\varepsilon}^{*}h=g_{\varepsilon}.$  Also, let $f^{*}h=g$ be the induced metric via $f$ on the complement of the branch points $\mathcal{B}.$
\item \label{imm equal} The maps $f_{\varepsilon}$ satisfy $f_{\varepsilon}=f$ on $\Sigma(\varepsilon)$ where, 
\begin{align}
\Sigma(\varepsilon)=\Sigma\backslash \{ \cup B_{g}(p_i, 3\varepsilon) \},
\end{align}
with the union taken over the set of branch points $\mathcal{B}.$ 
\item Let $K_{\varepsilon}$ be the sectional curvature of the metric $g_{\varepsilon}$ and $K_g$ be the sectional curvature of the metric $g$ on the complement of the branch points $\mathcal{B}.$  Then 
$K_{\varepsilon}\rightarrow K_g$ pointwise on the complement of $\mathcal{B}$ in $\Sigma.$  Note that this formally follows from the previous property.
\item  There exists $\kappa(\varepsilon_0)>0$ such that $K_{\varepsilon}(p)<-\kappa(\varepsilon_0)$ for all $\varepsilon\in (0,\varepsilon_{0}]$ and for all $p\in\Sigma.$
\end{enumerate} 
\end{proposition}

\begin{proof}
Let $\varepsilon>0$ and select $p_i\in \mathfrak{B}.$  First we need a bump function:
\begin{align}
\eta_{\varepsilon}(\lvert t \rvert)=\left\{
     \begin{array}{lr}
     1 &: \lvert t\rvert <\varepsilon \\
       0 &:  \lvert t \rvert >2\varepsilon 
     \end{array}
\right\}
   \end{align}
where also $0\leq \eta_{\varepsilon}\leq 1.$  We claim that there exists a function $q(t):\mathbb{R}\rightarrow \mathbb{R}$ such that $q(0)=0, \ q(t)>0$ if $t>0,$ and furthermore as $\varepsilon\rightarrow 0,$
\begin{align}\label{boundbump}
\lVert q(\varepsilon)\eta_{\varepsilon}(t)\rVert_{C^{3}(\mathbb{R})}\rightarrow 0.
\end{align}
This can be achieved by choosing,
\begin{align}
q(t)=e^{-\frac{1}{t}}.
\end{align}
Pick a coordinate chart on the ball of radius $3\varepsilon$ (choosing $\varepsilon$ small enough so that $p_i$ is the only branch point in the chart) about $p_i$ sending $p_i$ to $0.$  Also, choose a coordinate chart of small radius about the image $f(p_i)$ sending $f(p_i)$ to $0.$  In these coordinates write 
\begin{align}
f:B_{\mathbb{R}^2}(0,r)&\rightarrow B_{\mathbb{R}^n}(0, r)\\
f(u,v)&\mapsto(f^1(u,v),...,f^n(u,v)).
\end{align}
where $r>0$ is some small number on which the coordinate chart is defined.
By the normal form for a branched immersion near a branched point (see definition \ref{def: branch}) we may assume that the first derivatives of $f^1$ and $f^2$ have an isolated zero at $(0,0).$
Next, define a perturbation of $f$ via,
\begin{align}
f_{\varepsilon}(u,v)=f(u,v)+\big(0,0,...,q(\varepsilon)\cdot\eta_{\varepsilon}(\lvert (u,v) \rvert)\cdot u,\ q(\varepsilon)\cdot \eta_{\varepsilon}(\lvert (u,v) \rvert)\cdot v\big).
\end{align}
Note that this is where we use that the dimension of $X$ is at least $4.$
Here, we equip the image of a small ball about $p_i$ with the induced Riemannian metric so that the coordinate chart is a local isometry, and $\lvert (u,v)\rvert$ is the distance from $0$ to $(u,v).$  Observe that $f_{0}=f.$  Furthermore, since $f$ has a branch point at $p_i,$
\begin{align}
\frac{\partial f_\varepsilon}{\partial u}\vert_{(0,0)}&=(0,0,...,q(\varepsilon),0), \\
\frac{\partial f_\varepsilon}{\partial v}\vert_{(0,0)}&=(0,0,...,0,q(\varepsilon)).
\end{align}
The first thing to observe is that $f_{\varepsilon}$ is now an immersion at $(0,0).$  Next, again by the normal form for branched immersions, the projection of $f$ onto the first two factors,
\begin{align}
\pi\circ f(u,v)=(f^1(u,v),f^2(u,v)),
\end{align}
is an immersion on $B(0,r')\backslash \{0\}$ for some $0<r'<r.$  Choosing $2\varepsilon=\frac{r'}{2},$ it follows from the definition of our bump function $\eta_{\varepsilon}$ that $f_{\varepsilon}=f$ on the complement of $B\left(0,\frac{r'}{2}\right).$  But, $f$ is already known to be an immersion on $B(0,r')\backslash \{0\},$ and thus $f_{\varepsilon}$ is an immersion everywhere.  Repeating this process at each branch point produces an immersion $f_{\varepsilon}:\Sigma\rightarrow X/\Gamma$ such that $f_{\varepsilon}=f$ outside of the union of the $2\varepsilon$-neighborhoods of the branch points.  This takes care of all the points except the last.

By \eqref{boundbump}, the convergence
\begin{align}
\lVert f_{\varepsilon}-f\rVert_{C^3(\Sigma)}\rightarrow 0
\end{align}
as $\varepsilon\rightarrow 0$ is immediate.  Let $g_{\varepsilon}$ be the induced metric via the immersion $f_{\varepsilon}.$  On the punctured surface $\hat{\Sigma},$ there is $\text{C}^2$-convergence of the Riemannian metrics
\begin{align}
(\hat{\Sigma},g_{\varepsilon})\rightarrow (\hat{\Sigma},g).
\end{align}
and hence $C^2$-convergence of the associated volume elements,
\begin{align}
dV_{g_{\varepsilon}}\rightarrow dV_{g}.
\end{align}
Since $g$ has sectional curvature less than $-1,$ for all $p\in \hat{\Sigma}$ and for all $r>0$ small enough,
\begin{align}
\text{Vol}_{g}(B_{g}(p,r))\geq \pi r^2+\beta r^4,
\end{align}
for $\beta>0$ some constant.
Since the volume elements converge, it follows that for $\varepsilon>0$ small enough, and for all $p\in\hat{\Sigma},$ and all $r>0$ small enough,
\begin{align}
\text{Vol}_{g_{\varepsilon}}(B_{g_{\varepsilon}}(p,r))\geq \pi r^2 +\beta' r^4,
\end{align}
for some smaller constant $\beta'>0.$
The Taylor expansion of the volume of balls (see \cite{GRA04}) implies that there exists some $\kappa>0$ such that the scalar curvature satisfies $R_{g_{\varepsilon}}(p)<-\kappa$ for all $\varepsilon>0$ small enough and all $p\in \hat{\Sigma}.$  Since we are on a surface, this implies that the sectional curvature of $g_{\varepsilon}$ satisfy,
\begin{align}
K_{g_{\varepsilon}}<-\kappa',
\end{align}
for some $\kappa'>0.$  This completes the proof.
\end{proof}

\bibliography{DDBib}{}
\bibliographystyle{alpha}

\end{document}